\documentclass[10pt,final]{siamltex}
\usepackage{amsmath}
\usepackage{bm}
\usepackage{amssymb,version}
\usepackage{cases}
\usepackage{pifont}
\usepackage{color}
\usepackage{verbatim}
\usepackage{rotating}
\usepackage{multirow}
\usepackage{graphicx}
\newtheorem{rem}{Remark}[section]
\usepackage{subfigure}
\usepackage{float}
\allowdisplaybreaks
\usepackage{silence}
\usepackage{hyperref}
\begin{document}
\graphicspath{{figures/},}
    \title{
    A new class of efficient linear higher-order schemes \\
    for the Landau-Lifshitz-Gilbert equation with reduced restriction on the damping parameter
\thanks{This work is supported in part by the National Natural Science Foundation of China (Grant Nos.12271302, 12131014, W2431008, 12501554 and 12371409) and Shandong Provincial Natural Science Foundation for Outstanding Youth Scholar (Grant No. ZR2024JQ030)}}
 \author{Fukeng Huang
        \thanks{School of Mathematical Sciences, Eastern Institute of Technology, Ningbo, China. Email: fkhuang@eitech.edu.cn}.
        \and Binghong Li
        \thanks{School of Mathematics and Mathematical Research Center, Shandong University, Jinan, Shandong, 250100, P.R. China. Email: binghongsdu@163.com}.
        \and Xiaoli Li
        \thanks{School of Mathematics and State Key Laboratory of Cryptography and Digital Economy Security, Shandong University, Jinan, Shandong, 250100, P.R. China. Email: xiaolimath@sdu.edu.cn}.
         \and Jie Shen
        \thanks{School of Mathematical Sciences, Eastern Institute of Technology, Ningbo, China. Email: jshen@eitech.edu.cn}.
}
\WarningFilter{latex}{No positions in optional float specifier.}
\maketitle
\begin{abstract} 
	Classical high-order backward differentiation formula (BDF) methods for the Landau-Lifshitz-Gilbert (LLG) equation often suffer from restrictive stability constraints, requiring small time steps and imposing stringent lower bounds on the damping parameter. These limitations become particularly severe for schemes of order higher than three.
	In this paper, we develop a class of high-order generalized BDF (GBDF) schemes for the LLG equation, including both semi-implicit and fully explicit treatments of the gyromagnetic term. The proposed schemes significantly improve stability properties and substantially relax the damping parameter constraints, but introduce essential difficulty in its analysis compared to the classical BDF schemes.
	We construct  a novel multiplier  which  enables us to carry out a  energy-based error analysis. This approach yields optimal-order error estimates under considerably weaker assumptions on the damping parameter than those required for classical BDF schemes. 
		Numerical experiments are presented to  confirm the theoretical results, and demonstrate that the proposed GBDF schemes achieve higher accuracy, enhanced stability, and much wider admissible damping regimes compared to classical high-order BDF methods.

\end{abstract}


  
 \section{Introduction}
 The Landau-Lifshitz-Gilbert (LLG) equation plays a fundamental role in the modeling of magnetization dynamics in ferromagnetic materials and has been widely used in the study of magnetic storage devices, spintronics, and topological magnetic structures \cite{jia2025electrically, jiang2015blowing, nagaosa2013topological, romming2013writing}. 
 It consists of a gyromagnetic precession term and a damping term \cite{landau1935theory}, and can be written as follows:
 \begin{equation}\label{e_original model}
 	\frac{\partial \textbf{m}}{\partial t} = -\beta \textbf{m} \times \Delta \textbf{m}
 	- \gamma \textbf{m} \times ( \textbf{m} \times \Delta \textbf{m} ), \quad  \ {\rm in} \ \Omega \times J, 
 \end{equation}
 with 
 \begin{equation}\label{e_initial condition}
 	\textbf{m}(\textbf{x},0 ) = \textbf{m}_0(\textbf{x}) , \text{ with } |\textbf{m}_0(\textbf{x}) |=1, \ {\rm in}\ \Omega,
 \end{equation} 
 subject to either homogeneous Neumann or periodic boundary conditions.
 In the above, $\textbf{m}=(m_1,m_2,m_3)^T$ describes the magnetization in continuum ferromagnets, $\Omega$ is an open bounded domain in $\mathbb{R}^d$ with $d \in \{ 1,2,3\}$,  
 $J$ denotes $(0, T]$ for some $T >0$, 
 $\gamma>0$ is the Gilbert damping parameter and $\beta \neq 0$ is an exchange parameter. The solution of \eqref{e_original model} preserves pointwisely its magnitude, i.e.,
 \begin{equation}\label{e_length preserving}
 	\aligned
 	\frac{d}{dt} | \textbf{m}(\textbf{x},t)  |^2=0,
 	\endaligned
 \end{equation} 
 which, together with  \eqref{e_initial condition},   implies that $ | \textbf{m}(\textbf{x},t )|=1$.
 
 From a mathematical perspective, the LLG equation is a highly nonlinear, constrained evolution equation characterized by a pointwise unit-length condition. These structural features pose significant challenges for both the design and analysis of accurate and stable numerical methods \cite{cohen1989relaxation, kim2017mimetic, prohl2001computational, suess2002time}.
 \textcolor{black}{Methods for the non-convex constraint $|\mathbf{m}| = 1$ can be found in \cite{cimrak2007survey, kruzik2006recent}, among which the normalization method $\mathbf{m}^n = \widetilde{\mathbf{m}}^n/|\widetilde{\mathbf{m}}^n|$ is the simplest to implement.} Since it is generally difficult to deal with the highly nonlinear term $\textbf{m} \times ( \textbf{m} \times \Delta \textbf{m} )$, we first rewrite the LLG equation \eqref{e_original model} in the following equivalent form:
 \begin{numcases}{}
 	\frac{\partial \textbf{m} }{\partial t} = \gamma \Delta\textbf{m} + \gamma |\nabla \textbf{m} |^2 \textbf{m} - \beta \textbf{m}\times  \Delta\textbf{m},  \label{e_model_transform1}\\
 	| \textbf{m} (x,t)|=1\quad\forall x,\, t. \label{e_model_transform3}
 \end{numcases}  
 

Due to the strong nonlinearity and the intrinsic magnitude constraint in the LLG equation \eqref{e_model_transform1}-\eqref{e_model_transform3}, the design of robust numerical schemes and the development of rigorous error analyses remain highly challenging. Considerable efforts have been devoted to this problem over the past decades. For spatial discretization, widely used approaches include finite difference methods \cite{an2021optimal,cai2022second,xie2025error} and finite element methods \cite{akrivis2021higher,alouges2006convergence,an2025optimal,gui2022convergence}. For temporal discretization, common strategies involve fully explicit, fully implicit, and implicit–explicit (IMEX) schemes for the gyromagnetic term $m\times\Delta m$ \cite{akrivis2021higher,li2026class} and Runge–Kutta methods \cite{cervera2007numerical,gui2024implicit}. Among these, IMEX schemes have attracted particular attention due to their computational efficiency and ease of implementation.

Despite these advances, the construction and rigorous analysis of high-order time-stepping schemes for the LLG equation remain challenging. A central difficulty lies in the interaction between the nonlinear constraint and the skew-symmetric structure of the gyromagnetic term. Existing analyses of high-order BDF-type schemes show that, for orders higher than three, stability and convergence typically require restrictive lower bounds on the damping parameter $\gamma$. These constraints become increasingly severe for fourth- and fifth-order methods, significantly limiting the practical applicability of high-order schemes, especially in regimes with small damping \cite{akrivis2021higher,he2024temporal,huang2003high}. However, the regimes with small damping are particularly important in practical applications as the damping term  affects both energy dissipation and response time \cite{cai2022second}. Thus, it of critical importance for a numerical scheme to allow the damping parameter $\gamma$ as small as possible.

The origin of this difficulty can be traced to the treatment of the gyromagnetic term 
$m\times \Delta m$. When this term is handled semi-implicitly or explicitly, the standard energy techniques used in the analysis of multistep schemes encounter a fundamental obstacle: the skew-symmetric trilinear form 
$b(u,v,w)=(u\times v,w)$, which satisfies 
$b(u,v,v)=0$, no longer yields sufficient cancellation at the discrete level. In particular, within the framework of high-order explicit or semi-explicit methods, the discrete trilinear terms fail to vanish and cannot be controlled by the available dissipative mechanisms. As a consequence, existing approaches rely on the damping term to compensate for this loss of structure, leading to the aforementioned restrictive conditions.

The purpose of this paper is to overcome this fundamental limitation by employing the generalized backward differentiation formula (GBDF) method, recently introduced in \cite{huang2024new} and was shown   to possess improved stability properties compared to classical BDF schemes, to construct a class of high-order semi-implicit and fully explicit schemes for the LLG equation that exhibit significantly improved stability properties and to carry out a rigorous error analysis.
The key ingredient of our analysis is the introduction of a novel multiplier, which restores a suitable structure in the discrete energy analysis and allows us to effectively control the non-vanishing trilinear terms.

This new multiplier leads to a unified analytical framework that applies to both semi-implicit and fully explicit GBDF schemes. As a result, we are able to establish optimal-order error estimates under substantially relaxed conditions on the damping parameter, in sharp contrast with existing high-order BDF analyses. In particular, for fourth- and fifth-order schemes, the proposed approach significantly enlarges the admissible parameter regimes and enables the use of larger time steps.

Our main contributions in this paper are summarized as follows:
\begin{itemize}
 
\item We construct high-order semi-implicit and fully explicit GBDF schemes for the LLG equation, which exhibit improved stability and broader applicability compared to existing methods.

\item We establish a key lemma (Lemma 3.3) with  a novel multiplier
which  enables us to establish optimal-order error estimates  under significantly relaxed damping constraints, particularly for fourth- and fifth-order schemes as shown in Tables \ref{semi-im} and \ref{explicit}.

\item We provide numerical experiments that validate the theoretical results and demonstrate clear advantages over classical high-order BDF methods.
\end{itemize}
The remainder of the paper is organized as follows. Section 2 introduces preliminary results and the norm-preserving GBDF schemes. Section 3 presents the new multiplier and its applicability to second- through fourth-order schemes. Section 4 establishes optimal-order error estimates for both semi-implicit and fully explicit treatments of the gyromagnetic term 
$m\times \Delta m$. Section 5 discusses the extension to fifth-order schemes. Section 6 provides numerical experiments demonstrating the advantages of the proposed methods over classical BDF schemes. Finally, concluding remarks are given in the last section.

\begin{table}[]
    \centering
    \caption{Semi-implicit schemes: theoretical minimum values of $\gamma$ to ensure stability}
    \begin{tabular}{c|c|c|c}
    \hline
    \multicolumn{4}{c}{Minimum value of $\gamma$ with $\beta
     = 1$} \\
    \hline
        order & Classical BDF \cite{akrivis2021higher, li2026class} &GBDF & $n+s$\\
        \hline
         3-order& 0.0913  & 0.0862 & s = 3\\
         \hline
         4-order& 0.4041  & 0.2504 & s = 4\\
         \hline
         5-order& 4.4348  & 0.7500 & s = 5\\
         \hline
    \end{tabular}
    \label{semi-im}
\end{table}
\begin{table}[]
    \centering
    \caption{Fully explicit schemes: theoretical minimum values of $\gamma$ to ensure stability}
    \begin{tabular}{c|c|c|c}
    \hline
    \multicolumn{4}{c}{Minimum value of $\gamma$ with $\beta
     = 1$} \\
    \hline
        order & Classical BDF \cite{li2026class,xie2025error}&GBDF & $n+s$\\
        \hline
         3-order& 7.00   & 2.00 & s = 3\\
         \hline
         4-order& 27.12  & 2.33 & s = 4\\
         \hline
         5-order& 306.0 & 5.00 & s = 5\\
         \hline
    \end{tabular}
    \label{explicit}
\end{table}
  \section{Some preliminaries and norm-preserving GBDF schemes}
In this section, we describe some notations and results, which will be frequently used in this paper. Furthermore, we present the length-preserving semi-implicit and fully explicit GBDF schemes.

We use the standard notations $L^2(\Omega)$, $H^k(\Omega)$ and $W^{k,p}(\Omega)$ to denote the usual Sobolev spaces over $\Omega$. The norm corresponding to $H^k(\Omega)$ will be denoted simply by $\|\cdot\|_k$. In particular, $\|\cdot\|$ and $(\cdot,\cdot)$ are used to denote the norm and the inner product in $L^2(\Omega)$, respectively. The vectors and vector spaces will be indicated by boldface type. Throughout the paper, we use $C$, with or without subscript, to denote a positive constant, which could have different values at different appearances.

 The following lemmas will be frequently used in the error analysis.
 \medskip
 \begin{lemma}(H\"older inequality) \label{lem: Holder inequality}
 Let $p, q, s >0$ such that  $\frac{1}{p}+\frac{1}{q}+\frac{1}{s}=1$. Then for vector functions $\textbf{u} \in \textbf{L}^p(\Omega)$, $\textbf{v}\in \textbf{L}^q(\Omega)$, and scalar function $w \in L^s(\Omega)$,  we have
 \begin{equation}\label{e_Preliminaries}
 \int_{\Omega} | (\textbf{u}, \textbf{v}) w | d \textbf{x} \leq \| \textbf{u} \|_{\textbf{L}^p} \| \textbf{v} \|_{\textbf{L}^q} \| w \|_{L^s}.
\end{equation}
\end{lemma}
 
  \begin{lemma}(Interpolation inequalities)  \label{lem: Interpolation inequalities}
  	For any $\textbf{f}$, then there exists a positive
  		constant $C$ such that
 \begin{equation}\label{e_Preliminaries2}
\| \textbf{f} \|_{\textbf{L}^k} \leq C \| \textbf{f} \|_{\textbf{L}^2}^{ \frac{6-k}{ 2k } } \| \textbf{f} \|_{\textbf{H}^1}^{ \frac{3k-6}{ 2k } }, \quad 3\le k\le 6,
\end{equation}
and
 \begin{equation}\label{e_Preliminaries3}
\aligned
\| \textbf{f} \|_{\textbf{L}^{\infty}} \leq C \| \textbf{f} \|_{\textbf{H}^1}^{ \frac{1}{ 2 } } \| \textbf{f} \|_{\textbf{H}^2}^{ \frac{1}{ 2 } }.
\endaligned
\end{equation}
\end{lemma}

We will frequently use the following discrete version of the Gronwall lemma (see \cite{HeSu07, shen1990long}).

\medskip
\begin{lemma} \label{lem: gronwall2}
Let $a_k$, $b_k$, $c_k$, $d_k$, $\gamma_k$, $\Delta t_k$ be non negative real numbers such that
\begin{equation}\label{e_Gronwall3}
\aligned
a_{k+1}-a_k+b_{k+1}\Delta t_{k+1}+c_{k+1}\Delta t_{k+1}-c_k\Delta t_k\leq a_kd_k\Delta t_k+\gamma_{k+1}\Delta t_{k+1}
\endaligned
\end{equation}
for all $0\leq k\leq m$. Then
 \begin{equation}\label{e_Gronwall4}
\aligned
a_{m+1}+\sum_{k=0}^{m+1}b_k\Delta t_k \leq \exp \left(\sum_{k=0}^md_k\Delta t_k \right)\{a_0+(b_0+c_0)\Delta t_0+\sum_{k=1}^{m+1}\gamma_k\Delta t_k \}.
\endaligned
\end{equation}
\end{lemma}

We recall some notations presented in \cite{huang2024new}, which are necessary for the constructed schemes and error analysis.

Set
$$\Delta t=T/N,\ t^n=n\Delta t,
\ {\rm for} \ n\leq N,$$
and
\begin{equation}\label{GBDF-n}
   \aligned
    A^s_k(\phi^n) = \sum^k_{q=0}a_{k,q} \phi^{n-k+q},\ \ B^s_k(\phi^n) = \sum^{k - 1}_{q=0}b_{k,q}\phi^{n-k+1+q},\ \ C^s_k(\phi^n) = \sum^{k - 1}_{q=0}c_{k,q}\phi^{n-k+1+q},
    \endaligned
\end{equation}
with $a_{k,q}, b_{k,q}, c_{k.q}$ defined as:
For $k = 2$,
\begin{subequations}
    \begin{align}
        &a_{2,2}(s) = \frac{2s + 1}{2},\,a_{2,1}(s) = -2s,\,a_{2,0}(s) = \frac{2s - 1}{2}, \\
        &b_{2,1}(s) = s,\,b_{2,0}(s) = - (s - 1),\,c_{2,1}(s) = s + 1,\,c_{2,0}(s) = - s.
    \end{align}
\end{subequations}
For $k = 3$,
\begin{subequations}
    \begin{align}
        &a_{3,3}(s) = \frac{3s^2 + 6s + 2}{6},\,a_{3,2}(s) = -\frac{9s^2+12s-3}{6},\,a_{3,1}(s) = \frac{9s^2+6s-6}{6}, \\
        &a_{3,0}(s) = -\frac{3s^2-1}{6},\,b_{3,2}(s) = \frac{s^2+s}{2},\,b_{3,1}(s) = - (s^2 - 1),\,b_{3,0}(s) = \frac{s^2 - s}{2},\\
        &c_{3,2}(s) = \frac{s^2 + 3s + 2}{2},\,c_{3,1}(s) = - (s^2 + 2s),\,c_{3,0}(s) = \frac{s^2 + s}{2}.
    \end{align}
\end{subequations}
For $k = 4$,
\begin{subequations}
    \begin{align}
        &a_{4,4}(s) = \frac{2s^3 + 9s^2 + 11s+3}{12},\,a_{4,3}(s) = \frac{-8s^3-30s^2-20s+10}{12},\\
        &a_{4,2}(s) = \frac{12s^3+36s^2+6s-18}{12},\,a_{4,1}(s) = \frac{-8s^3-18s^2+4s+6}{12},\\
        &a_{4,0}(s) = \frac{2s^3+3s^2 - s - 1}{12},\,b_{4,3}(s) = \frac{s^3+3s^2+2s}{6},\,b_{4,2}(s) = \frac{-s^3-2s^2+s+2}{2},\\
        &b_{4,1}(s) = \frac{s^3 + s^2 - 2s}{2},\,b_{4,0}(s) = \frac{-s^3 + s}{6},\,c_{4,3}(s) = \frac{s^3 + 6s^2 + 11s + 6}{6},\\
        &c_{4,2}(s) = \frac{-s^3 - 5s^2 - 6s}{2},\,c_{4,1}(s) = \frac{s^3 + 4s^2+3s}{2},\,c_{4,0}(s) = \frac{-s^3-3s^2-2s}{6}.
    \end{align}
\end{subequations}
And then we have the following truncation error,
\begin{equation}\label{GBDF}
    \aligned
    \frac{1}{\Delta t} A^{s}_k(\phi^{n+1}) &= \partial_t \phi^{n + s} + O(\Delta t^k),\\
    B^{s}_k(\phi^{n+1}) = \phi^{n + s} + &O(\Delta t^k),\,
    C^{s}_k(\phi^{n}) = \phi^{n + s} + O(\Delta t^k).
    \endaligned
\end{equation}

 
 Now we construct higher-order and norm preserving schemes for the LLG equation \eqref{e_model_transform1} based on the GBDF-$k$ formula. We construct below semi-implicit and explicit GBDF schemes.

Assuming $\textbf{m}^j$, $\widetilde{\textbf{m}}^{j}$ with $j=n,\,n-1,\ldots, \,n-k+1$ are given, we can solve $\widetilde{\textbf{m}}^{n+1}$ by using either 

\textbf{Scheme \uppercase\expandafter{\romannumeral1}} (Semi-implicit scheme)
\begin{align}\label{e_High-order2}
\displaystyle
\frac{1}{\Delta t}A^s_k( \widetilde{\textbf{m}}^{n+1}) =  \gamma \Delta B^s_k(\widetilde{\textbf{m}}^{n+1})+ \gamma  | \nabla C_k^s( \widetilde{\textbf{m}}^{n})  |^2 C_k^s( \textbf{m}^{n} ) - \beta C_k^s( \textbf{m}^{n} ) \times \Delta B_k^s ( \widetilde{\textbf{m}}^{n+1} ),
\end{align}
or 

\textbf{Scheme \uppercase\expandafter{\romannumeral2}} (Fully explicit scheme)
\begin{align}\label{e_High-order1}
\displaystyle
\frac{1}{\Delta t}A^s_k( \widetilde{\textbf{m}}^{n+1}) = \gamma \Delta B^s_k(\widetilde{\textbf{m}}^{n+1})+ \gamma  | \nabla C_k^s( \widetilde{\textbf{m}}^{n})  |^2 C_k^s( \textbf{m}^{n} ) - \beta C_k^s( \textbf{m}^{n} ) \times \Delta C_k^s ( \widetilde{\textbf{m}}^{n} ),
\end{align}
 where $A_k^s$ denotes the GBDF-$k$ formula, $B_k^s$ is the $k$-th order implicit difference operator, and $C_k^s$ is the $k$-th order explicit difference formula. Furthermore, we update $\mathbf{m}^{n+1}$ by 
 \begin{equation}\label{mcor}
     \displaystyle
     \textbf{m}^{n+1} =\frac{ \widetilde{\textbf{m}}^{n+1} } { | \widetilde{\textbf{m}}^{n+1}| }.
\end{equation}
At each time step, \textbf{Scheme I} requires solving a coupled linear system with variable coefficients, while  \textbf{Scheme II} only needs to solve  decoupled Poisson type equations. 

 \section{New Multiplier for the semi-implicit and fully explicit GBDF schemes}
In order to conduct the stability and error analysis for the  semi-implicit and fully explicit GBDF schemes by using energy method, a key step is to find a suitable multiplier. In this section, we propose a new multiplier and show that it is suitable for the GBDF schemes of second to fourth order.

To the end, we first split $B^s_k(\phi^{i})$ into two parts: 
\begin{equation}
    B^s_k(\phi^{i}) = \alpha_k(s)C^s_k(\phi^{i}) + D^s_k(\phi^i),
\end{equation}
with
\begin{equation}
    \alpha_2(s) = \frac{s - 1}{s},~\alpha_3(s) = \frac{s - 1}{s + 1},~\alpha_4(s) = \frac{s - 1}{s + 3},
\end{equation}

and recall the following result established in \cite{huang2024new} for GBDF method.
\medskip
\begin{lemma} \label{lem: multiplier step}
\cite{huang2024new} For $2 \leq k \leq 4$, a positive definite symmetric matrix $G_1=(g^1_{i,j}),~ G_2=(g^2_{i,j})\in \mathbb{R}^{k,k} $, and real numbers $ \delta^1_0,\ldots, \delta^1_k$ such that
\begin{align}\label{e_multiplier step}
( A_k^s (\phi^{n+1}) , C_k^s(\phi^{n+1})) = & \sum\limits_{i,j=1}^k g^1_{i,j}
( \phi^{n+1+i-k } ,  \phi^{n+1+j-k} ) \notag \\
& - \sum\limits_{i,j=1}^k g^1_{i,j}
( \phi^{n+i-k} ,  \phi^{n+j-k} )  + \|  \sum\limits_{i=0}^k  \delta^1_i \phi^{n+1+i-k} \|^2,
\end{align}
and real numbers $ \delta^2_0,\ldots, \delta^2_k$ such that
\begin{align}
( D_k^s (\phi^{n+1}) , C_k^s(\phi^{n+1})) = & \sum\limits_{i,j=1}^k g^2_{i,j}
( \phi^{n+1+i-k } ,  \phi^{n+1+j-k} ) \notag \\
& - \sum\limits_{i,j=1}^k g^2_{i,j}
( \phi^{n+i-k} ,  \phi^{n+j-k} )  + \|  \sum\limits_{i=0}^k  \delta^2_i \phi^{n+1+i-k} \|^2,
\end{align} 
where $s \geq 1$ if $k = 2,~3$ and $s \geq 2$ if $k = 4$.
\end{lemma}

\begin{rem}
The above lemma effectively relaxes the constraint on the damping parameter for the fully explicit GBDF schemes. However, a direct application of this lemma to the semi-implicit GBDF schemes does not yield a substantial improvement over the existing results, \textcolor{black}{the main reason being that
\begin{equation}
    (C^s_k(\mathbf{m}^n)\times \Delta B^s_k(\widetilde{\mathbf{m}}^{n+1}), \Delta C^s_k(\widetilde{\mathbf{m}}^{n+1})) =  (C^s_k(\mathbf{m}^n)\times \Delta D^s_k(\widetilde{\mathbf{m}}^{n+1}), \Delta C^s_k(\widetilde{\mathbf{m}}^{n+1}))
\end{equation}
which would significantly strengthen the requirements on $\gamma$ and the relevant analysis techniques can be found in \cite{huang2025stability}.} Therefore, in this paper, we introduce a new multiplier (see Lemma \ref{GBDFlem}), which leads to a significant relaxation of the restriction on the damping parameter.
\end{rem}

First, we recall the result from Dahlquist's G-stability theory \cite{dahlquist1978g}.
\medskip
\begin{lemma}\label{definite G}
    Let $\eta(\zeta) = \eta_k\zeta^k+\cdots+\eta_0$ and $\mu(\zeta) = \mu_k\zeta^k+\cdots +\mu_0$ be polynomials of degree at most $k$ (and at least one of them of degree $k$) that have no common divisors. Let $(\cdot,\cdot)$ be an inner product with associated norm $|\cdot|$. If
    \begin{equation}\label{Re>0}
      Re\frac{\eta(\zeta)}{\mu(\zeta)} > 0,~~for~|\zeta| >1,
    \end{equation}
    then there exists a symmetric positive definite matrix $G=(g_{i,j}) \in \mathbb{R}^{k \times k}$ and real $\delta_0,\cdots,\delta_k$ such that for $v^0,\cdots,v^k$ in the inner product space,
    \begin{equation}\label{G-conclusion}
        \left(\sum^k_{i=0}\eta_iv^i,\sum^k_{i=0}\mu_jv^j\right) = \sum^k_{i,j=0}g_{i,j}(v^i,v^j) - \sum^k_{i,j=0}g_{i,j}(v^{i-1},v^{j-1}) + |\sum^k_{i=0}\delta_iv^i|^2.
    \end{equation}
\end{lemma}
We set
\begin{equation}
    \aligned
    \widetilde{A}^s_k(\zeta) :=\sum^k_{q=0} a_{k,q}\zeta^q,\, \ \widetilde{Z}^s_k(\zeta): = \sum^{k-1}_{q=0} (b_{k,q} + \tau_k c_{k,q}) \zeta^{q+1} =: \sum^{k}_{q=0} z_{k,q}\zeta^{q}.
    \endaligned
\end{equation}
\begin{table}[]
    \centering
    \caption{Multiplier Comparison}
    \begin{tabular}{c|c|c}
    \hline
        Method & Multiplier &stability term \\
        \hline
         \textbf{Classical BDF}& $\phi^{n+1} - \tau'_k \phi^{n}$  & $\|\nabla \phi^{n+1}\|_{L^2}$ \\
         \hline
         \multirow{2}{*}{{\textbf{GBDF}}}& $C^s_k(\phi^{n+1})$  & $\|\nabla C^s_k(\phi^{n+1})\|_{L^2}$ \\
         & $B^s_k(\phi^{n+1}) + \tau_k C^s_k(\phi^{n+1})$  &$\|\nabla C^s_k(\phi^{n+1})\|_{L^2},\,\|\nabla D^s_k(\phi^{n+1})\|_{L^2}$ \\
         \hline
    \end{tabular}
    \label{Multiplier Comparison}
\end{table}
 \textcolor{black}{A comparison of multipliers for the different methods is presented in Table \ref{Multiplier Comparison}}. Inspired by the classical BDF and Lemma \ref{lem: multiplier step}, we propose a new multiplier for the semi-implicit schemes which allows us to prove the following key lemma.
 \medskip
 \begin{lemma}\label{GBDFlem}
     Given $s \geq 1$, there exist symmetric positive definite matrix $H_k = (h^k_{i,j}) \in \mathbb{R}^{k\times k}$ and $\tau_k$ for $k = 2,\,3,\,4$ such that
     \begin{equation}\label{GBDFk}
     \aligned
         &\left(A^s_k(\phi^{n+1}),B^s_k(\phi^{n+1}) + \tau_kC^s_k(\phi^{n+1})\right) \\
         &= \sum^k_{i,j=1}h^k_{i,j}(\phi^{n-k+1+i},\phi^{n-k+1+j}) - \sum^k_{i,j=1}h^k_{i,j}(\phi^{n-k+i},\phi^{n-k+j}) +\|\sum^k_{i=0}\theta_i \phi^{n-k+1+i}\|^2,
     \endaligned
     \end{equation}
     where $\theta_0,\cdots,\theta_k$ are real and the values of $\tau_k$ are
     $$\tau_2 = 0,~\tau_3 = 0.094,~\tau_4 = 0.275.$$ 
 \end{lemma}
\begin{proof}
    It is easy to check that
    \begin{equation}
        \aligned
        (A^s_2(\phi^{n+1}),B^s_2(\phi^{n+1})) = &a_2\|\phi^{n+1}\|^2-a_2\|\phi^{n}\|^2+\|b_2\phi^{n+1}+c_2\phi^{n}\|^2 \\
        &-\|b_2\phi^{n}+c_2\phi^{n-1}\|^2 + \|d_2\phi^{n+1}+e_2\phi^{n}+f_2\phi^{n-1}\|^2,
        \endaligned
    \end{equation}
    where $\displaystyle e_2 = -\sqrt{2s^2-s},~c_2 = d_2 =f_2 = -\frac{e_2}{2},~b_2 =-\frac{1}{2}\sqrt{\frac{2s-1}{s}}(s+1),~a_2 = \frac{1}{4s}$. 
    
    Next we focus on the case of $k=3,4$.

    \textbf{Case \uppercase\expandafter{\romannumeral 1}}: $k = 3$. First we show $gcd(\widetilde{A}^s_3(\zeta),\widetilde{B}^s_3(\zeta)+\tau_3\zeta\widetilde{C}^s_3(\zeta)) = 1$, i.e. they have no common divisor by using $Sylvester~resultant$ as follows. The Sylvester matrix of  $\widetilde{A}^s_3(\zeta)$ and $\widetilde{B}^s_3(\zeta) + \tau_k\zeta\widetilde{C}^s_3(\zeta)$ is
    \begin{equation}
        Sly(\widetilde{A}^s_3,\widetilde{Z}^s_3) = 
        \left[\begin{array}{cccccc}
                 a_{3,3} &a_{3,2} &a_{3,1} &a_{3,0} &0 &0\\
                 0 &a_{3,3} &a_{3,2} &a_{3,1} &a_{3,0} &0\\
                 0 &0 &a_{3,3} &a_{3,2} &a_{3,1} &a_{3,0}\\
                 z_{3,3} &z_{3,2} &z_{3,1} &0 &0 &0\\
                 0 &z_{3,3} &z_{3,2} &z_{3,1} &0 &0\\
                 0 &0 &z_{3,3} &z_{3,2} &z_{3,1} &0
              \end{array}
        \right],
    \end{equation}
    \textcolor{black}{due to $z_{3,0} = 0$}. It is easy to verify that its determinant is
    \begin{equation}
        \aligned
        det(Sly(\widetilde{A}^s_3,\widetilde{Z}^s_3)) \neq 0,~for~s \geq 1,
        \endaligned
    \end{equation}
    which implies that $gcd(\widetilde{A}^s_3(\zeta),\widetilde{B}^s_3(\zeta)+\tau_3\zeta\widetilde{C}^s_3(\zeta)) = 1$.

    Next, we show $\frac{\widetilde{A}^s_3(\zeta)}{\widetilde{Z}^s_3(\zeta)}$ is holomorphic outside the unit disk in the complex plane. Actually, we only need to show that all three zeros of $\widetilde{Z}^s_3(\zeta)$ are inside the unit disk. Note that
    \begin{equation}
        \aligned
          \widetilde{Z}^s_3(x) = (z_{3,3}x^2+z_{3,2}x+z_{3,1})x
        \endaligned
    \end{equation}
    with
    \begin{equation}
    \aligned
        \Delta = z_{3,2}^2 - 4z_{3,1}z_{3,3} = -\frac{299209}{250000}s^2-\frac{25709}{125000}s + 1<0.
    \endaligned
    \end{equation}
    Therefore, $\widetilde{Z}^s_3(x) = 0$ has exactly one real root $x_1=0$ and two complex roots, denoted as $x_2,\,x_3 = \bar{x}_2$ such as
    \begin{equation}
        \aligned
        |x_2|^2 = x_2x_3 = \frac{z_{3,1}}{z_{3,3}} = \frac{s(547s-453)}{547s^2+641s+94} <1.
        \endaligned
    \end{equation}
    As a result, we have $|x_1|,~|x_2|,~|x_3| <1$ and hence $\frac{\widetilde{A}^s_3(\zeta)}{\widetilde{Z}^s_3(\zeta)}$ is holomorphic outside the unit disk.

    On the other hand, we have
    \begin{equation}
        \lim_{|\zeta|\to +\infty}\frac{\widetilde{A}^s_3(\zeta)}{\widetilde{Z}^s_3(\zeta)} = \frac{a_{3,3}}{z_{3,3}} = \frac{1500s^2+3000s+1000}{1641s^2+1923s+282} >0.
    \end{equation}
    Therefore, it follows from the maximum principle for harmonic functions that $Re\frac{\widetilde{A}^s_3(\zeta)}{\widetilde{Z}^s_3(\zeta)}>0$ for all $|\zeta| >1$ is equivalent to 
    \begin{equation}
        Re[\widetilde{A}^s_3(e^{i\theta})\widetilde{Z}^s_3(e^{-i\theta})]\geq0,~\theta \in[0,2\pi).
    \end{equation}
    Letting $y =\cos \theta$ and using Chebyshev polynomials
    \begin{equation}\label{Chebyshev}
        \cos 2\theta= 2y^2-1,~\cos3\theta = 4y^3-3y,~\cos4\theta = 8y^4-8y^2+1,
    \end{equation}
    we can easily have that
    \begin{equation}
        Re(\widetilde{A}^s_3(e^{i\theta})\widetilde{Z}^s_3(e^{-i\theta})) = (1-y)(\omega_2(s)y^2 + \omega_1(s)y + \omega_0(s)) =: (1-y)f(y)
    \end{equation}
    with
    \begin{equation}
        \aligned
          w_2(s) &= 4(\frac{s^2}{2} - \frac{1}{6})(\frac{547}{1000}s^2 + \frac{641}{1000}s + \frac{47}{500}) >0,\\
          w_1(s) &= -\frac{547}{250}s^4-\frac{641}{250}s^3+\frac{1889}{1500}s^2+\frac{1547}{1500}s-\frac{26}{125} <0,\\
          w_0(s) &= \frac{547}{500}s^4+\frac{641}{500}s^3-\frac{406}{375}s^2-\frac{151}{250}s+\frac{172}{375} >0.
        \endaligned
    \end{equation}
    It is clear that 
    \begin{align}
       \omega_2 >0,~~\Delta = \omega_1^2 - 4\omega_2\omega_0 = -\frac{463 s^4}{250000}-\frac{80493s^3}{125000} -\frac{197971s^2}{450000}+\frac{28628s}{140625}+\frac{22252}{140625}<0.     \end{align}
    which means $f(y)>0$ for all $y \in [-1,1]$. Therefore, we obtain $Re(\widetilde{A}^s_3(e^{i\theta})\widetilde{Z}^s_3(e^{-i\theta}))\geq 0$ and \eqref{GBDFk} for $k=3$ by using Lemma~\ref{definite G}.
    
    \textbf{Case \uppercase\expandafter{\romannumeral 2}}: $k = 4$. First, it is easy to verify that its determinant is
\begin{equation}
\aligned
det(Sly(\widetilde{A}^s_4,\widetilde{Z}^s_4))\neq 0,~for~s\geq2,
\endaligned
\end{equation}
which implies that $gcd(\widetilde{A}^s_4(\zeta),\widetilde{Z}^s_4(\zeta)) = 1$ by using the $Sylvester$ resultant.

Next, we show $\frac{\widetilde{A}^s_4(\zeta)}{\widetilde{Z}^s_4(\zeta)}$ is holomorphic outside the unit disk. To this end, it suffices to show that all three zeros of $\widetilde{Z}^s_4(\zeta)$ are inside the unit disk. Note that
\begin{equation}
    \widetilde{Z}^s_4(x) = x(z_{4,4}x^3+z_{4,3}x^2+z_{4,2}x+z_{4,1})=:xh(x),
\end{equation}
and
\begin{equation}
    h'(x) = 3z_{4,4}x^2+2z_{4,3}x+z_{4,2} 
\end{equation}
with 
\begin{equation}
    \aligned
    z_{4,4} &= \frac{637}{3000}s^3+\frac{387}{500}s^2+\frac{2507}{3000}s+\frac{137}{500} > 0,\\
    \Delta &= -\frac{405769}{250000}s^4-\frac{246519}{31250}s^3-\frac{2651149}{250000}s^2-\frac{79921}{125000}+4 <0,
    \endaligned
\end{equation}
which means $h(x)$ is monotonically increasing in the real axis. Recalling
\begin{equation}
    h(0) = -\frac{s(637s^2+411s-226)}{3000} <0, \ h(1) =1.274>0.
\end{equation}
Therefore, $h(x)$ has exactly one real root, denoted as $x_1\in (0, 1)$, and two complex roots, denoted as $x_2,~x_3$, in the complex plane. Next we find that
\begin{equation}
    h(-\frac{z_{4,1}}{z_{4,4}}) = \frac{s(-516949706s^3-408609383s^2+138305440s+25450764)}{500(637s^2+1685s+822)^2}<0,
\end{equation}
which implies $\displaystyle x_1 \in(-\frac{z_{4,1}}{z_{4,4}},\,1)$.
By Vieta's formula, we have that
\begin{equation}
    |x_2|^2 =x_2x_3 = -\frac{z_{4,1}}{z_{4,4}}\frac{1}{x_1} < 1.
\end{equation}
As a result, we have $|x_1|,|x_2|,|x_3|<1$ and hence $\frac{\widetilde{A}^s_4(\zeta)}{\widetilde{Z}^s_4(\zeta)}$ are holomorphic outside the unit disk.

On the other hand, we can easily have that
\begin{equation}
    \lim_{|\zeta|\to +\infty}\frac{\widetilde{A}^s_3(\zeta)}{\widetilde{Z}^s_3(\zeta)} = \frac{a_{4,4}}{z_{4,4}} >0.
\end{equation}
Therefore, it follows from the maximum principle for harmonic function that $Re(\frac{\widetilde{A}^s_4(\zeta)}{\widetilde{Z}^s_4(\zeta)}) > 0$ for all $|\zeta|>1$ is equivalent to
\begin{equation}\label{Re_4}
    Re[\widetilde{A}^s_4(e^{i\theta})\widetilde{Z}^s_4(e^{-i\theta})] \geq 0.
\end{equation}
Letting $y=cos \theta$ and using \eqref{Chebyshev}, we find that
\begin{equation}
    \aligned
    Re(\widetilde{A}^s_4(e^{i\theta})\widetilde{Z}^s_4(e^{-i\theta})) = (1-y)(\omega_3y^3 + \omega_2y^2+\omega_1y+\omega_0)=:(1-y)f(y)
    \endaligned
\end{equation}
with
\begin{align}
    w_3 &= 8(-\frac{s^3}{6}-\frac{s^2}{4}+\frac{s}{12}+\frac{1}{12})(\frac{637}{3000}s^3+\frac{387}{500}s^2+\frac{2507}{3000}s+\frac{137}{500}) <0,\notag\\
    w_2 &= \frac{637}{750}s^6+\frac{437}{100}s^5+\frac{10343}{1500}s^4+\frac{2051}{900}s^3-\frac{863}{375}s^2-\frac{2579}{2250}s-\frac{4}{125}>0,\label{w-de4}\\
    w_1 &= -\frac{637}{750}s^6-\frac{437}{100}s^5-\frac{9343}{1500}s^4-\frac{473}{1125}s^3+\frac{4541}{1500}s^2+\frac{329}{4500}s-\frac{589}{1500} < 0,\notag\\
    w_0 &= \frac{637}{2250}s^6 + \frac{437}{300}s^5 +\frac{927}{500}s^4-\frac{719}{1500}s^3-\frac{563}{450}s^2+\frac{s}{3}+\frac{79}{100}>0.\notag
\end{align}
It is clear that \eqref{Re_4} is equivalent to 
\begin{equation}\label{Re_4*}
    \aligned
    f(y)\geq0,~\forall~y \in[-1,1].
    \endaligned
\end{equation}
With $\omega$ defined in \eqref{w-de4} and $s \geq 2$, we have
\begin{align}
    f(1) >0,~f(-1) >0,
\end{align}
and
\begin{equation}
    f'(y) = 3\omega_3y^2 + 2\omega_2y+\omega_1.
\end{equation}
If $f'(y)$ does not have zero in $[-1,1]$, then we can get \eqref{Re_4*}. Otherwise, supposing that there exists $-1 \leq y \leq 1$ such that $f'(y_0) = 0$, we only need to show $f(y_0)\geq 0$. Indeed, with $f'(y_0) = 0$, we have
\begin{equation}
    \aligned
    3f(y_0)= 3f(y_0) - y_0f'(y_0) = \omega_2y_0^2+2\omega_1y_0+3\omega_0,
    \endaligned
\end{equation}
then
\begin{equation}
    \aligned
    3\omega_3f(y_0)&= \omega_3( \omega_2y_0^2+2\omega_1y_0+3\omega_0) - \frac{\omega_2}{3}f'(y_0) \\
    &= (2\omega_1\omega_3-\frac{2\omega_2^2}{3})y_0 +3\omega_0\omega_3-\frac{\omega_1\omega_2}{3}.
    \endaligned
\end{equation}
We set $\displaystyle y_1 = -\frac{9\omega_0\omega_3-\omega_1\omega_2}{6\omega_1\omega_3-2\omega_2^2}$ and find that
\begin{equation}
    \aligned
      f'(1)>0,~f'(y_1) < 0,~6\omega_1\omega_3 - 2 \omega_2^2 <0.
    \endaligned
\end{equation}
Therefore, $y_0 \in (y_1,1)$ and $3\omega_3f(y_0)<0$. Hence $f(y_0)>0$ implies \eqref{Re_4*} with $\omega_3<0$. Furthermore, we obtain $Re(\widetilde{A}^s_4(e^{i\theta})\widetilde{Z}^s_4(e^{-i\theta}))\geq 0$ and \eqref{GBDFk} for $k=4$ by using Lemma~\ref{definite G}.
\end{proof}
\section{Error estimate}
In this section, we carry out a rigorous error analysis for \eqref{e_High-order2} and \eqref{e_High-order1}. We denote $\mathbf{m}_e$ as the exact solution for the PDE system \eqref{e_original model} and 
\begin{equation}
    \displaystyle \widetilde{e}_{\textbf{m}}^{n+1}= \textbf{m}^{n+1}_e - \widetilde{\textbf{m}}^{n+1},\quad \displaystyle e_{\textbf{m}}^{n+1}= \textbf{m}^{n+1}_e - \textbf{m}^{n+1}.
\end{equation}

Before performing error analysis, it is essential to establish the relationship between the error terms $e_{\textbf{m}}^{l} $ and \(\widetilde{e}_{\textbf{m}}^{l}\). It is a well-established result in the normalization method (see \cite[Lemma 5]{li2026class}). 
\medskip
\begin{lemma}\label{lem: relation for error functions}
 Assuming that $\|\widetilde{e}^l_\textbf{m}\|_{L^\infty} < \epsilon'$ hold for $2 \leq p \leq 4$, then we have 
\begin{equation}\label{e_relation_L2}
\aligned
& \| e_{\textbf{m}}^{l} \|_{L^{\infty}} \leq C \| \widetilde{e}_{\textbf{m}}^{l} \|_{L^{\infty}},~\| e_{\textbf{m}}^{l} \|_{W^{1,p}} \leq C \| \widetilde{e}_{\textbf{m}}^{l} \|_{W^{1,p}},
\endaligned
\end{equation}
where $C$ is a positive constant  independent of $\Delta t$ and $\epsilon'$ can be chosen as an arbitrarily small positive constant to satisfy $\|\widetilde{\textbf{m}}^l\|_{L^{\infty}} \geq 1/2$.
\end{lemma}
\subsection{Error analysis for the semi-implicit scheme}
Let $\textbf{R}_{\textbf{m}, k,s}^{l+1}$ be the truncation error and recalling \eqref{GBDF}, we can obtain that
\begin{equation}\label{e_error6}
\aligned
\textbf{R}_{\textbf{m},k,s }^{l+1}=& \frac{1}{\Delta t}A^s_k(\textbf{m}_e^{l+1})-\partial_t \textbf{m}^{l+s}_e -\gamma(\Delta B^s_k(\textbf{m}_e^{l+1}) - \Delta\textbf{m}_e^{l+s}) -\gamma(|\nabla C^s_k(\textbf{m}^{l}_{e})|^2C^s_k(\textbf{m}^l_e)\\
& - |\nabla \textbf{m}^{l+s}_e|^2\textbf{m}^{l+s}_e) + \beta(\textbf{m}_e^{l+s} \times \Delta B^s_k(\textbf{m}^{l+1}_e) - \textbf{m}^{l+s}_e\times\Delta \textbf{m}^{l+s}_e)=O(\Delta t^k).
\endaligned
\end{equation}

Subtracting \eqref{e_High-order1}  from \eqref{e_model_transform1} at $t^{l+s}$, we obtain an error equation  corresponding to \eqref{e_High-order2}
\begin{equation}\label{e_error7}
\aligned
 \frac{1}{\Delta t}A^s_k(\widetilde{e}^{l+1}_{\textbf{m}}) -  \gamma \Delta B^s_k(\widetilde{e}_{\textbf{m}}^{l+1})
= & \gamma \left( |\nabla C^s_k(\textbf{m}^{l}_e) |^2 C^s_k(\textbf{m}_e^{l})  -   |\nabla C^s_k( \widetilde{\textbf{m}}^{l} ) |^2 C^s_k( \textbf{m}^{l} )\right) + \textbf{R}_{\textbf{m},k,s }^{l+1} \\
& - \beta \left(  \textbf{m}^{l+s}_e \times  \Delta B^s_k(\textbf{m}^{l+1}_e) - C^s_k( \textbf{m}^{l} )  \times  \Delta B^s_k ( \widetilde{\textbf{m}}^{l+1} ) \right).
\endaligned
\end{equation}

Note that
\begin{equation}
    \aligned
    &\left(\Delta  B^s_k(\widetilde{e}_{\textbf{m}}^{l+1}),\Delta  B^s_k(\widetilde{e}_{\textbf{m}}^{l+1})+\tau_k\Delta  C^s_k(\widetilde{e}_{\textbf{m}}^{l+1})\right) \\
    &= (\alpha_k^2+\tau_k\alpha_k)\| \Delta C^s_k(\widetilde{e}_{\textbf{m}}^{l+1}) \|^2+\|\Delta D^s_k(\widetilde{e}_{\textbf{m}}^{l+1})\|^2 + (2\alpha_k+\tau_k)(\Delta  C^s_k(\widetilde{e}_{\textbf{m}}^{l+1}),\Delta  D^s_k(\widetilde{e}_{\textbf{m}}^{l+1})).
    \endaligned
\end{equation}

Taking the inner product of \eqref{e_error7} with $ - \Delta t (\Delta  B^s_k(\widetilde{e}_{\textbf{m}}^{l+1})+\tau_k\Delta  C^s_k(\widetilde{e}_{\textbf{m}}^{l+1}))$ and using Lemma~\ref{lem: multiplier step}
and Lemma~\ref{GBDFlem}, we obtain
\begin{align}
& \sum\limits_{i,j=1}^k h_{i,j}
( \nabla \widetilde{e}_{\textbf{m}}^{l+1+i-k } ,  \nabla \widetilde{e}_{\textbf{m}}^{l+1+j-k } ) - \sum\limits_{i,j=1}^k h_{i,j}
( \nabla \widetilde{e}_{\textbf{m}}^{l+i-k} ,  \nabla \widetilde{e}_{\textbf{m}}^{l+j-k} ) \notag \\
& +  \Delta t \gamma \bigg( (\alpha_k(s)^2+\tau_k\alpha_k(s))\| \Delta C^s_k(\widetilde{e}_{\textbf{m}}^{l+1}) \|^2+\|\Delta D^s_k(\widetilde{e}_{\textbf{m}}^{l+1})\|^2  \notag\\ & +(2\alpha_k + \tau_k)\sum\limits_{i,j=1}^k g^2_{i,j}(
( \Delta \widetilde{e}_{\textbf{m}}^{l+1+i-k }) ,  \Delta \widetilde{e}_{\textbf{m}}^{l+1+j-k } ) 
- ( \Delta \widetilde{e}_{\textbf{m}}^{l+i-k} ,  \Delta \widetilde{e}_{\textbf{m}}^{l+j-k} ) \bigg)\notag\\  
\leq &  \Delta t \gamma \left(  |\nabla C^s_k( \widetilde{\textbf{m}}^{l} ) |^2 -   |\nabla C^s_k( \textbf{m}^l_e ) |^2 ,  C^s_k( \textbf{m}^{l} )  (\Delta  B^s_k(\widetilde{e}_{\textbf{m}}^{l+1})+\tau_k\Delta  C^s_k(\widetilde{e}_{\textbf{m}}^{l+1})) \right) \notag \\
& -  \Delta t \gamma \left(|\nabla C^s_k( \textbf{m}^l_e ) |^2 C^s_k ( e_{\textbf{m}}^{l} ),  \Delta  B^s_k(\widetilde{e}_{\textbf{m}}^{l+1})+\tau_k\Delta  C^s_k(\widetilde{e}_{\textbf{m}}^{l+1})\right) \label{e_error8} \\
& + \textcolor{black}{ \Delta t \beta \left( \textbf{m}_e^{l+s}  \times  \Delta B^s_k ( \widetilde{e}_\textbf{m}^{l+1} ) , 
 \Delta  B^s_k(\widetilde{e}_{\textbf{m}}^{l+1})+\tau_k\Delta  C^s_k(\widetilde{e}_{\textbf{m}}^{l+1})\right) } \notag\\
& + \textcolor{black}{ \Delta t \beta \left(( \textbf{m}^{l+s}_e - C^s_k( \textbf{m}^{l}_e ))  \times  \Delta B^s_k ( \textbf{m}_e^{l+1} ) , 
 \Delta  B^s_k(\widetilde{e}_{\textbf{m}}^{l+1})+\tau_k\Delta  C^s_k(\widetilde{e}_{\textbf{m}}^{l+1})\right) } \notag\\
 &- \textcolor{black}{ \Delta t \beta \left(( \textbf{m}^{l+s}_e - C^s_k( \textbf{m}^{l}_e ))  \times  \Delta B^s_k (\widetilde{e}^{l+1}_\textbf{m}) , 
 \Delta  B^s_k(\widetilde{e}_{\textbf{m}}^{l+1})+\tau_k\Delta  C^s_k(\widetilde{e}_{\textbf{m}}^{l+1})\right) } \notag\\
 & + \textcolor{black}{ \Delta t \beta \left( C^s_k( e_\textbf{m}^{l} )  \times  \Delta B^s_k ( \textbf{m}_e^{l+1} ) , 
 \Delta  B^s_k(\widetilde{e}_{\textbf{m}}^{l+1})+\tau_k\Delta  C^s_k(\widetilde{e}_{\textbf{m}}^{l+1})\right) } \notag\\
  & - \textcolor{black}{ \Delta t \beta \left( C^s_k( e_\textbf{m}^{l} )  \times  \Delta B^s_k (\widetilde{e}^{l+1}_\textbf{m}), 
 \Delta  B^s_k(\widetilde{e}_{\textbf{m}}^{l+1})+\tau_k\Delta  C^s_k(\widetilde{e}_{\textbf{m}}^{l+1})\right) } \notag\\
 & +  \Delta t \left( \textbf{R}_{\textbf{m},k,s}^{l+1}, \Delta  B^s_k(\widetilde{e}_{\textbf{m}}^{l+1})+\tau_k\Delta  C^s_k(\widetilde{e}_{\textbf{m}}^{l+1})\right)  \notag \\
=: &\Delta t( I_1+ I_2 + I_3 + I_4 + I_5 + I_6 + I_7 + I_8).  \notag
\end{align}

Next we derive the main error estimate for the semi-implicit scheme \eqref{e_High-order2}.
\medskip
\begin{theorem}\label{the: error_estimate_final_semi}
 Supposing that the damping parameter $\gamma$ satisfies
  \begin{equation} \label{e_beta_value2}
\gamma > \frac{ |\beta | \tau_k }{ 2\sqrt{\alpha_k(s)^2+\tau_k\alpha_k(s)}}
  \end{equation}
 and assuming  $\textbf{m}\in H^{k+1}(0,T;\textbf{L}^2(\Omega))\bigcap H^{k}(0,T;\textbf{H}^{2}(\Omega)) \bigcap L^{\infty}(0,T;\textbf{W}^{2,\,3}(\Omega))$, then for the semi-implicit scheme \eqref{e_High-order2}, we have
 \begin{equation}  \label{e_final_semi}
\aligned
\|  e_\textbf{m}^{n+1}\|^2 + \|  \nabla e_\textbf{m}^{n+1}\|^2  \leq  C(\Delta t)^{2k}, \ \forall\,0\leq n\leq N-1,\ 2\le k\le 4,
\endaligned
\end{equation} 
where $C$ is a positive constant independent of $\Delta t$ and $\Delta t \leq \epsilon_1$ with positive constant $\epsilon_1$ referred in \eqref{epsilon1}.
\end{theorem}

\begin{proof}
We shall first prove, by an induction process with a bootstrap argument
\begin{equation}\label{e_error2}
\aligned
\| \widetilde{e}_{\textbf{m}}^{n}  \|_{H^2}  \leq (\Delta t)^{1/4} , \ \forall \,n \leq N.
\endaligned
\end{equation}

Obviously \eqref{e_error2} holds for $n=0$. Now we suppose 
\begin{equation}\label{e_error4}
\aligned
\| \widetilde{e}_{\textbf{m}}^{l}  \|_{H^2} \leq (\Delta t)^{1/4} , \ \forall \,l \leq n,
\endaligned
\end{equation}
and prove below that $ \| \widetilde{e}_{\textbf{m}}^{n+1}  \|_{H^2}  \leq (\Delta t)^{ 1/4}  $ hold true. We would like to highlight that Lemma \ref{lem: relation for error functions} can be applied if \(\Delta t < (\epsilon')^4\).

\noindent{\bf Step 1: Error estimate for the term $\widetilde{e}_{\mathbf{m}}$ in $H^2$-norm.} 

We bound the above eight terms in \eqref{e_error8} as follows. Using the H\"older inequality and the interpolation inequality in Lemma \ref{lem: Holder inequality} and Lemma \ref{lem: relation for error functions}, the first term can be estimated by
\begin{align}
 | I_1 |  = &  \left|\gamma \left( \nabla C^s_k( \widetilde{e}_{\textbf{m}}^{l} ) \nabla C^s_k( \textbf{m}^l_e + \widetilde{\textbf{m}}^l )  , C^s_k( \textbf{m}^{l} )  (\Delta  B^s_k(\widetilde{e}_{\textbf{m}}^{l+1})+\tau_k\Delta  C^s_k(\widetilde{e}_{\textbf{m}}^{l+1}))   \right) \right|\notag  \\
\leq & \gamma \| \nabla C^s_k( \widetilde{e}_{\textbf{m}}^{l} ) \|_{L^3}\|\nabla C^s_k( \textbf{m}^l_e + \widetilde{\textbf{m}}^l )\|_{L^6} \| C^s_k( \textbf{m}^{l}) (\Delta  B^s_k(\widetilde{e}_{\textbf{m}}^{l+1})+\tau_k\Delta  C^s_k(\widetilde{e}_{\textbf{m}}^{l+1})) \|_{L^2} \notag\\
\leq &C \| \nabla C^s_k( \widetilde{e}_{\textbf{m}}^{l} )  \|_{L^3}   \|\Delta  B^s_k(\widetilde{e}_{\textbf{m}}^{l+1})+\tau_k\Delta  C^s_k(\widetilde{e}_{\textbf{m}}^{l+1}) \|_{L^2}, \label{e_error_gamma1}\\
\leq&  C  \| \nabla C^s_k( \widetilde{e}_{\textbf{m}}^{l} ) \|_{L^2}^{1/2}   \| \Delta C^s_k( \widetilde{e}_{\textbf{m}}^{l} )\|_{L^2}^{1/2}  \| \Delta  B^s_k(\widetilde{e}_{\textbf{m}}^{l+1})+\tau_k\Delta  C^s_k(\widetilde{e}_{\textbf{m}}^{l+1}) \|_{L^2} \notag  \\
\leq &  \epsilon  \|  \Delta  C^s_k(\widetilde{e}_{\textbf{m}}^{l+1}) \|_{L^2}^2 + \epsilon  \| \Delta C^s_k(\widetilde{e}_{\textbf{m}}^{l}) \|_{L^2}^2  +\epsilon  \|  \Delta  D^s_k(\widetilde{e}_{\textbf{m}}^{l+1}) \|_{L^2}^2 + C  \|  \nabla  C^s_k( \widetilde{e}_{\textbf{m}}^{l} ) \|_{L^2}^2 , \notag 
\end{align}
where $\displaystyle \|\nabla C^s_k(\textbf{m}^l)\|_{L^6} \leq \displaystyle \|\nabla C^s_k(\textbf{m}^l_e)\|_{L^6} + C\|\nabla C^s_k(\widetilde{e}_\textbf{m}^l)\|_{H^1} \leq C$ and  $\epsilon$ is an arbitrarily small positive constant which is independent of $\Delta t$ and will be determined below.

By using the H\"older  inequality and \eqref{e_relation_L2} , the second term on the right-hand side of \eqref{e_error8} can be bounded by 
\begin{align}
 | I_2 |  = &  \gamma |  \left(  | \nabla C^s_k( \textbf{m}^l_e ) |^2 C^s_k ( e_{\textbf{m}}^{l} ),   \Delta  B^s_k(\widetilde{e}_{\textbf{m}}^{l+1})+\tau_k\Delta  C^s_k(\widetilde{e}_{\textbf{m}}^{l+1}) \right)  |  \notag \\
\leq & \gamma \| | \nabla C^s_k( \textbf{m}^l_e ) |^2  \|_{L^3} \| C^s_k( e_{\textbf{m}}^{l} ) \|_{L^6} \|  \Delta  B^s_k(\widetilde{e}_{\textbf{m}}^{l+1})+\tau_k\Delta  C^s_k(\widetilde{e}_{\textbf{m}}^{l+1})\|_{L^2} \notag \\
\leq & C \| C^s_k( e_{\textbf{m}}^{l} ) \|_{H^1} \|  \Delta  B^s_k(\widetilde{e}_{\textbf{m}}^{l+1})+\tau_k\Delta  C^s_k(\widetilde{e}_{\textbf{m}}^{l+1}) \|_{L^2}  \label{e_error_gamma2}\\
\leq & \epsilon  \|  \Delta  C^s_k(\widetilde{e}_{\textbf{m}}^{l+1}) \|_{L^2}^2 + \epsilon  \|  \Delta  D^s_k(\widetilde{e}_{\textbf{m}}^{l+1}) \|_{L^2}^2 + C \sum^{k}_{i = 1}\| \widetilde{e}_{\textbf{m}}^{l + k - i}  \|^2 _{H^1} . \notag
\end{align}

The third term on the right-hand side of \eqref{e_error8} plays a crucial role in the estimate of the stability term, and it is  essential to keep this term as small as possible. By applying the Cauchy-Schwarz inequality, we have
\begin{align}\label{e_error_beta1}
|I_3| &= | \beta | \left|  \left(   \textbf{m}^{l+s}_e \times  \Delta B^s_k(\widetilde{e}_{\textbf{m}}^{l+1}) , \Delta  B^s_k(\widetilde{e}_{\textbf{m}}^{l+1}) + \tau_k \Delta C^s_k(  \widetilde{e}_{\textbf{m}}^{l+1}) \right)\right| \notag\\
&= | \beta | \left|  \left(   \textbf{m}^{l+s}_e \times  \Delta D^s_k(\widetilde{e}_{\textbf{m}}^{l+1}) , \tau_k \Delta C^s_k(  \widetilde{e}_{\textbf{m}}^{l+1}) \right)\right|\notag\\
&\leq  \frac{\tau_k|\beta|\sqrt{\alpha_k^2+\tau_k\alpha_k}}{2}\|\Delta C^s_k(\widetilde{e}_{\textbf{m}}^{l+1})\|^2_{L^2} + \frac{\tau_k|\beta|}{2\sqrt{\alpha_k^2+\tau_k\alpha_k}}\|\Delta D^s_k(\widetilde{e}^{l+1}_\textbf{m})\|^2_{L^2},
\end{align}
where we use the property of cross product $(a\times b, a)=(a\times b,b)=0,\, \forall a, b \in R^d$ and $B_k^s=\alpha_k C_k^s+D_k^s$ in the second equality.

For the fourth term on the right hand side of \eqref{e_error8} by applying Cauchy-Schwarz inequality and \eqref{e_error6}, we obtain
\begin{align}\label{e_error_beta2}
    |I_4| &= |\beta|\left|\left(( \textbf{m}^{l+s}_e - C^s_k( \textbf{m}^{l}_e ))  \times  \Delta B^s_k ( \textbf{m}_e^{l+1} ) , 
 \Delta  B^s_k(\widetilde{e}_{\textbf{m}}^{l+1})+\tau_k\Delta  C^s_k(\widetilde{e}_{\textbf{m}}^{l+1})\right) \right| \notag\\
 &\leq |\beta|\|\textbf{m}^{l+s}_e - C^s_k(\textbf{m}^l_e)\|_{L^6}\|\Delta B^s_k(\textbf{m}^{l+1}_e)\|_{L^3}\|\Delta  B^s_k(\widetilde{e}_{\textbf{m}}^{l+1})+\tau_k\Delta  C^s_k(\widetilde{e}_{\textbf{m}}^{l+1})\|_{L^2} \\
 &\leq \epsilon  \|  \Delta  C^s_k(\widetilde{e}_{\textbf{m}}^{l+1}) \|_{L^2}^2 + \epsilon  \|  \Delta  D^s_k(\widetilde{e}_{\textbf{m}}^{l+1}) \|_{L^2}^2 + C (\Delta t)^{2k} \notag
 \end{align}

 The fifth term on the right hand side of \eqref{e_error8} can be estimated by
 \begin{align}
     |I_5| &= |\beta|\left|\left(( \textbf{m}^{l+s}_e - C^s_k( \textbf{m}^{l}_e ))  \times  \Delta B^s_k (\widetilde{e}^{l+1}_\textbf{m}) , 
 \Delta  B^s_k(\widetilde{e}_{\textbf{m}}^{l+1})+\tau_k\Delta  C^s_k(\widetilde{e}_{\textbf{m}}^{l+1})\right)\right| \notag\\
 &\leq |\beta| \|\textbf{m}^{l+s}_e - C^s_k(\textbf{m}^l_e)\|_{L^\infty}\|\Delta B^s_k(\widetilde{e}^{l+1}_\textbf{m})\|_{L^2}\|\Delta  B^s_k(\widetilde{e}_{\textbf{m}}^{l+1})+\tau_k\Delta  C^s_k(\widetilde{e}_{\textbf{m}}^{l+1})\|_{L^2} \\
 &\leq C(\Delta t)^{k}\|\Delta B^s_k(\widetilde{e}^{l+1}_\textbf{m})\|_{L^2}\|\Delta  B^s_k(\widetilde{e}_{\textbf{m}}^{l+1})+\tau_k\Delta  C^s_k(\widetilde{e}_{\textbf{m}}^{l+1})\|_{L^2} \notag \\
 &\leq \epsilon  \|  \Delta  C^s_k(\widetilde{e}_{\textbf{m}}^{l+1}) \|_{L^2}^2 + \epsilon  \|  \Delta  D^s_k(\widetilde{e}_{\textbf{m}}^{l+1}) \|_{L^2}^2, \notag
 \end{align}
 with $\Delta t \leq (\epsilon/C_1)^{1/k}$.
 
 Applying Cauchy-Schwarz inequality and \eqref{e_relation_L2} , the sixth term on the right hand side of \eqref{e_error8} can be estimated by
 \begin{align}
     |I_6| &= |\beta|\left| \left( C^s_k( e_\textbf{m}^{l} )  \times  \Delta B^s_k ( \textbf{m}_e^{l+1} ) , 
 \Delta  B^s_k(\widetilde{e}_{\textbf{m}}^{l+1})+\tau_k\Delta  C^s_k(\widetilde{e}_{\textbf{m}}^{l+1})\right)\right| \notag\\
 &\leq |\beta|\|C^s_k(e^l_\textbf{m})\|_{L^6}\|\Delta B^s_k ( \textbf{m}_e^{l+1} )\|_{L^3}\|\Delta  B^s_k(\widetilde{e}_{\textbf{m}}^{l+1})+\tau_k\Delta  C^s_k(\widetilde{e}_{\textbf{m}}^{l+1})\|_{L^2}\\
 &\leq \epsilon  \|  \Delta  C^s_k(\widetilde{e}_{\textbf{m}}^{l+1}) \|_{L^2}^2 + \epsilon  \|  \Delta  D^s_k(\widetilde{e}_{\textbf{m}}^{l+1}) \|_{L^2}^2 + C \sum^{k}_{i = 1}\| \widetilde{e}_{\textbf{m}}^{l + k - i}  \|^2 _{H^1}. \notag
 \end{align}

 For the seventh term on the right hand side of \eqref{e_error8} by applying Cauchy-Schwarz inequality and \eqref{e_error6}, we have
 \begin{align}
     |I_7| &= |\beta|\left|\left( C^s_k( e_\textbf{m}^{l} )  \times  \Delta B^s_k (\widetilde{e}^{l+1}_\textbf{m}), 
 \Delta  B^s_k(\widetilde{e}_{\textbf{m}}^{l+1})+\tau_k\Delta  C^s_k(\widetilde{e}_{\textbf{m}}^{l+1})\right)\right| \notag\\
 &\leq |\beta|\|C^s_k(e^l_\textbf{m})\|_{L^\infty}\|\Delta B^s_k (\widetilde{e}^{l+1}_\textbf{m})\|_{L^2}\|\Delta  B^s_k(\widetilde{e}_{\textbf{m}}^{l+1})+\tau_k\Delta  C^s_k(\widetilde{e}_{\textbf{m}}^{l+1})\|_{L^2}\\
 &\leq C(\Delta t)^{1/4}\|\Delta B^s_k (\widetilde{e}^{l+1}_\textbf{m})\|_{L^2}\|\Delta  B^s_k(\widetilde{e}_{\textbf{m}}^{l+1})+\tau_k\Delta  C^s_k(\widetilde{e}_{\textbf{m}}^{l+1})\|_{L^2}\notag\\
 &\leq \epsilon  \|  \Delta  C^s_k(\widetilde{e}_{\textbf{m}}^{l+1}) \|_{L^2}^2 + \epsilon  \|  \Delta  D^s_k(\widetilde{e}_{\textbf{m}}^{l+1}) \|_{L^2}^2, \notag
 \end{align}
 with $\Delta t \leq (\epsilon/C_2)^{4}$.
 
Recalling \eqref{e_error6}, the final term on the right hand side of \eqref{e_error8} can be controlled by
  \begin{align}\label{e_error_R}
  | I_8 |  &= | ( \textbf{R}_{\textbf{m},k,s}^{l+1}, \Delta  B^s_k(\widetilde{e}_{\textbf{m}}^{l+1})+\tau_k\Delta  C^s_k(\widetilde{e}_{\textbf{m}}^{l+1})) | \notag \\
  &\leq    \epsilon  \|  \Delta  C^s_k(\widetilde{e}_{\textbf{m}}^{l+1} )\|_{L^2}^2 + \epsilon  \|  \Delta  D^s_k(\widetilde{e}_{\textbf{m}}^{l+1} )\|_{L^2}^2
   + C (\Delta t)^{2k}.
\end{align} 
It can be observed that the control term on the right-hand side depends on the $\|\widetilde{e}^{l+1}_\textbf{m}\|_{L^2}$. This issue can be addressed straightforwardly by taking the inner product of \eqref{e_error7} with \(\Delta t\,\widetilde{e}^{l+1}_\textbf{m}\).
Since the magnitude of the norm control is reduced at this point, the expected control term and the right-hand side term can be easily obtained, \textcolor{black}{without introducing uncontrollable $H^2$ terms}. Therefore by combining \eqref{e_error8} with \eqref{e_error_gamma1}-\eqref{e_error_R}, we obtain
\begin{align}
& \sum\limits_{i,j=1}^k h_{i,j}
( \nabla \widetilde{e}_{\textbf{m}}^{l+1+i-k } ,  \nabla \widetilde{e}_{\textbf{m}}^{l+1+j-k } ) - \sum\limits_{i,j=1}^k h_{i,j}
( \nabla \widetilde{e}_{\textbf{m}}^{l+i-k} ,  \nabla \widetilde{e}_{\textbf{m}}^{l+j-k} ) \notag \\
& +\sum\limits_{i,j=1}^k h_{i,j}
(  \widetilde{e}_{\textbf{m}}^{l+1+i-k } ,  \widetilde{e}_{\textbf{m}}^{l+1+j-k } ) - \sum\limits_{i,j=1}^k h_{i,j}
( \widetilde{e}_{\textbf{m}}^{l+i-k} ,  \widetilde{e}_{\textbf{m}}^{l+j-k} ) \notag \\
& +  \Delta t \gamma \bigg( (\alpha_k(s)^2+\tau_k\alpha_k(s))\| \Delta C^s_k(\widetilde{e}_{\textbf{m}}^{l+1}) \|^2+\|\Delta D^s_k(\widetilde{e}_{\textbf{m}}^{l+1})\|^2  \label{e_error14}\\ 
& +(2\alpha_k + \tau_k)\sum\limits_{i,j=1}^k g^2_{i,j}(
( \Delta \widetilde{e}_{\textbf{m}}^{l+1+i-k }) ,  \Delta \widetilde{e}_{\textbf{m}}^{l+1+j-k } ) 
- ( \Delta \widetilde{e}_{\textbf{m}}^{l+i-k} ,  \Delta \widetilde{e}_{\textbf{m}}^{l+j-k} ) \bigg) \notag \\
\leq &  (7 \epsilon +\frac{\tau_k|\beta|\sqrt{\alpha_k(s)^2+\tau_k\alpha_k(s)}}{2}) \Delta t \|  \Delta  C^s_k(\widetilde{e}_{\textbf{m}}^{l+1}) \|_{L^2}^2 + \epsilon \Delta t\|  \Delta  C^s_k(\widetilde{e}_{\textbf{m}}^{l}) \|_{L^2}^2 \notag\\
& + (7\epsilon + \frac{\tau_k|\beta|}{2\sqrt{\alpha_k(s)^2+\tau_k\alpha_k(s)}}) \Delta t \|  \Delta  D^s_k(\widetilde{e}_{\textbf{m}}^{l+1}) \|_{L^2}^2+ C \Delta t ((\Delta t )^{2k}  +\sum^{k}_{i = 1}\| \widetilde{e}_{\textbf{m}}^{l + k - i}  \|^2 _{H^1} )\notag.
\end{align}

Summing \eqref{e_error14} over $l$ from $k-1$ to $n$, and thanks to Lemma \ref{lem: multiplier step},  $H=(h_{i,j})$ and $G^2 = (g^2_{i,j})$ are symmetric positive definite matrices with minimum eigenvalues $\lambda^1_{\min},~\lambda^2_{\min}$ and $\lambda_{\min} = \min \{\lambda_{\min}^1,~\lambda_{\min}^2\}$, we obtain
{\color{black}  
  \begin{align}\label{e_error15}
& \left( (\alpha_k(s)^2+\tau_k\alpha_k(s))\gamma - \frac{\tau_k|\beta|\sqrt{\alpha_k(s)^2+\tau_k\alpha_k(s)}}{2}
-  8 \epsilon \right)  \Delta t  \sum\limits_{l=k-1}^{n} \| \Delta C^s_k(\widetilde{e}_{\textbf{m}}^{l+1}) \|^2 \notag\\
&\quad+(\gamma-\frac{\tau_k|\beta|}{2\sqrt{\alpha_k(s)^2+\tau_k\alpha_k(s)}}-7\epsilon)\Delta t \sum^n_{l=k-1} \|\Delta D^s_k(\widetilde{e}^{l+1}_{\mathbf{m}})\|^2\\
&\quad + \lambda_{\min} \|\widetilde{e}_{\textbf{m}}^{n+1 }\|_{H^1}^2+ \Delta t \gamma \lambda_{\min} \| \Delta \widetilde{e}_{\textbf{m}}^{n+1}\|^2 \leq  C \Delta t \sum\limits_{l=0}^{n-1}   \|\widetilde{e}_{\textbf{m}}^{l} \|_{H^1}^2 + C (\Delta t) ^{2k}.\notag
\end{align}

Using the condition \eqref{e_beta_value2} and setting 
$$ \displaystyle \epsilon = \min\left\{1,\,\sqrt{\alpha_k(s)^2+\tau_k\alpha_k(s)}\right\}(\gamma-\frac{\tau_k|\beta|}{2\sqrt{\alpha_k(s)^2+\tau_k\alpha_k(s)}}) /16,$$ 
}
and applying the discrete Gronwall Lemma \ref{lem: gronwall2} to \eqref{e_error15}, we can arrive at
 \begin{align}\label{e_error_final_H2}
\aligned
& \|\widetilde{e}_{\textbf{m}}^{n+1} \|_{H^1}^2 + \Delta t \| \Delta \widetilde{e}_{\textbf{m}}^{n+1}\|^2
\leq C_3 (\Delta t)^{2k},  k=2,3,4,
\endaligned
\end{align}  
where $C_3$ is independent of $\Delta t$ and step $n$.

\noindent{\bf Step 2: Completion of  the induction process.} 

Letting 
\begin{equation}\label{epsilon1}
    \displaystyle \Delta t \leq \min\{1,(\epsilon')^4,(\frac{\epsilon}{C_1})^{\frac{1}{k}},(\frac{\epsilon}{C_2})^4,(\frac{1}{C_3})^{\frac{1}{2k-2}}\} =: \epsilon_1, 
\end{equation} 
we can obtain
\begin{align}\label{e_error_final_induction6}
\|  \widetilde{e}_{\textbf{m}}^{n+1} \|_{H^1}  +  \| \Delta \widetilde{e}_{\textbf{m}}^{n+1} \| 
& \leq \sqrt{C_3 (\Delta t)^{2k - 2}} (\Delta t)^{1/2}  \leq (\Delta t)^{1/4} ,
\end{align} 
which completes the induction process \eqref{e_error2}.

Recalling Lemma \ref{lem: relation for error functions}, we have the final results
\begin{equation}  \label{e_final_m}
\aligned
\|  e_\textbf{m}^{n+1}\|^2 + \|  \nabla e_\textbf{m}^{n+1}\|^2  \leq  C(\Delta t)^{2k}, \ \forall\,0\leq n\leq N-1,\ 2\le k\le 4,
\endaligned
\end{equation} 
where $C$ is a positive constant  independent of $\Delta t$.
The proof of Theorem \ref{the: error_estimate_final_semi} is completed.
\end{proof}
\subsection{Error analysis for the fully explicit scheme}  

In this subsection, we establish error estimates for the scheme \eqref{e_High-order1}, where the term $\beta \textbf{m}\times \Delta \textbf{m}$ is treated explicitly under the condition $\gamma > | \beta |/  \alpha_k$. We first recall  \eqref{e_High-order1} 
\begin{equation*}
	\aligned
	\frac{1}{\Delta t}A_k^s (\widetilde{\textbf{m}}^{n+1}) = & \gamma \Delta B^s_k(\widetilde{\textbf{m}}^{n+1})+ \gamma | \nabla C^s_k( \widetilde{\textbf{m}}^{n})  |^2 C^s_k( \textbf{m}^{n} )- \beta C^s_k( \textbf{m}^{n} ) \times \Delta  C^s_k(\widetilde{\textbf{m}}^{n}).
	\endaligned
\end{equation*} 
\medskip
\begin{theorem}\label{the: error_estimate_final}
 Supposing that the damping parameter $\gamma$ satisfies $\gamma > | \beta |/  \alpha_k$, and the time step $\Delta t \leq \epsilon_2$ and assuming  $\textbf{m}\in H^{k+1}(0,T;\textbf{L}^2(\Omega))\bigcap H^{k}(0,T;\textbf{H}^{2}(\Omega)) \bigcap L^{\infty}(0,T;\textbf{W}^{2,\,3}(\Omega))$, then for the fully discrete scheme \eqref{e_High-order1}, we have
 \begin{equation}  \label{e_final_fully}
\aligned
\|  e_\textbf{m}^{n+1}\|^2 + \|  \nabla e_\textbf{m}^{n+1}\|^2  \leq  C(\Delta t)^{2k}, \ \forall\,0\leq n\leq N-1,\ 2\le k\le 4,
\endaligned
\end{equation} 
where $C$ and $\epsilon_2$ are two positive constants  independent of $\Delta t$.
\end{theorem}

\begin{proof}
The proof of Theorem \ref{the: error_estimate_final} is exactly the same as Theorem \ref{the: error_estimate_final_semi}, except for the highly nonlinear term with exchange parameter $\beta \ne 0$, so for the sake of brevity our focus will be on the this term.

At this point, we take the inner product of the error equation with \(-\Delta t\,\Delta C^s_k(\widetilde{e}^{l+1})\) and \(\Delta t\, C^s_k(\widetilde{e}^{l+1})\). Similarly to the error analysis for the semi-implicit scheme, we can obtain the estimates of \(I_i\) (\(1 \leq i \leq 8\)) with particular emphasis on the estimate of the third term and the coefficient of the left-hand side stability term, which are considered fundamental. By applying the Cauchy-Schwarz inequality, we have
\begin{align}\label{e_error_betafully}
|I_3| &= | \beta | \left|  \left(   \textbf{m}^{l+s}_e \times  \Delta C^s_k(\widetilde{e}_{\textbf{m}}^{l}) , \Delta C^s_k(  \widetilde{e}_{\textbf{m}}^{l+1}) \right)\right| \notag\\
&\leq  \frac{|\beta|}{2}\|\Delta C^s_k(\widetilde{e}_{\textbf{m}}^{l})\|^2_{L^2} + \frac{|\beta|}{2}\|\Delta C^s_k(\widetilde{e}^{l+1}_\textbf{m})\|^2_{L^2}.
\end{align}
Applying Lemma \ref{lem: multiplier step}, the stability term on the left-hand side leads to
\begin{align}
&\textcolor{black}{(A^s_k(\widetilde{e}^{l+1}_{\textbf{m}}) -  \gamma \Delta t \Delta B^s_k(\widetilde{e}_{\textbf{m}}^{l+1}),-\Delta C^s_k(\widetilde{e}_{\textbf{m}}^{l+1})) + (A^s_k(\widetilde{e}^{l+1}_{\textbf{m}}), C^s_k(\widetilde{e}_{\textbf{m}}^{l+1}))}\notag\\
&\textcolor{black}{=(A^s_k(\widetilde{e}^{l+1}_{\textbf{m}}) -  \gamma \Delta t   \Delta(\alpha_k C^s_k(\widetilde{e}_{\textbf{m}}^{l+1}) +  D^s_k(\widetilde{e}_{\textbf{m}}^{l+1})),-\Delta C^s_k(\widetilde{e}_{\textbf{m}}^{l+1})) + (A^s_k(\widetilde{e}^{l+1}_{\textbf{m}}), C^s_k(\widetilde{e}_{\textbf{m}}^{l+1}))}\notag\\
&= \sum\limits_{i,j=1}^k g^1_{i,j}
(( \nabla \widetilde{e}_{\textbf{m}}^{l+1+i-k } ,  \nabla \widetilde{e}_{\textbf{m}}^{l+1+j-k } ) -
( \nabla \widetilde{e}_{\textbf{m}}^{l+i-k} ,  \nabla \widetilde{e}_{\textbf{m}}^{l+j-k} ) )\\
& +\sum\limits_{i,j=1}^k g^1_{i,j}
((  \widetilde{e}_{\textbf{m}}^{l+1+i-k } ,  \widetilde{e}_{\textbf{m}}^{l+1+j-k } ) - 
( \widetilde{e}_{\textbf{m}}^{l+i-k} ,  \widetilde{e}_{\textbf{m}}^{l+j-k} ) ) +  \Delta t \gamma \bigg( \alpha_k(s)\| \Delta C^s_k(\widetilde{e}_{\textbf{m}}^{l+1}) \|^2  \notag \\ 
& +\sum\limits_{i,j=1}^k g^2_{i,j}(
( \Delta \widetilde{e}_{\textbf{m}}^{l+1+i-k }) ,  \Delta \widetilde{e}_{\textbf{m}}^{l+1+j-k } ) 
- ( \Delta \widetilde{e}_{\textbf{m}}^{l+i-k} ,  \Delta \widetilde{e}_{\textbf{m}}^{l+j-k} ) \bigg).\notag
\end{align}

The subsequent analysis proceeds in a manner analogous to that of Theorem \ref{the: error_estimate_final_semi},  with the principal distinction that the condition \(\gamma > |\beta| /\alpha_k\) should be satisfied.
\end{proof}


\section{Extension to fifth order accuracy}
In Lemma \ref{GBDFlem}, we identify appropriate multipliers for the second-, third- and fourth-order schemes with $s \geq 2$. In this section, we aim to demonstrate the suitable multiplier for the fifth-order scheme with a fixed $s = 5$.

\begin{lemma}\label{GBDFlem-5}
     Given $s = 5$, there exists symmetric positive definite matrix $H = (h_{i,j}) \in \mathbb{R}^{5\times 5}$ and $\tau_5$ such that
     \begin{equation}\label{GBDF5}
     \aligned
         &\left(A^s_5(\phi^{n+1}),B^s_5(\phi^{n+1}) + \tau_5 C^s_5(\phi^{n+1})\right) \\
         &= \sum^5_{i,j=1}h_{i,j}(\phi^{n-5+1+i},\phi^{n-5+1+j}) - \sum^5_{i,j=1}h_{i,j}(\phi^{n-5+i},\phi^{n-5+j}) +\|\sum^5_{i=0}\theta_i \phi^{n-5+1+i}\|^2,
     \endaligned
     \end{equation}
     where $\theta_0,\cdots,\theta_5$ are real and $\tau_5 = 0.6$.
 \end{lemma}
 \begin{proof}
Following the same notation as previously established, the coefficients can be obtained from \cite{huang2024new}. Additionally, the term 
 $B^s_5(\phi^{n+1})$ can be decomposed as follows
     \begin{equation}\label{GBDF_5}
         B^s_5(\phi^{n+1}) = \alpha_5(s)C^s_5(\phi^{n+1}) + D^s_5(\phi^{n+1}),~~\rm{with}~\alpha_5(s) = \frac{s-1}{s + 15}.
     \end{equation}
     We first show $gcd(\widetilde{A}^s_5(\zeta), \widetilde{Z}^s_5(\zeta)) = 1$ by using the $Sylvester~resultant$ as follows
     \begin{equation}
         det(Sly(\widetilde{A}^s_5,\widetilde{Z}^s_5)) \neq 0,~for~s = 5.
     \end{equation}
     Next we check $\frac{\widetilde{A}^s_5(\zeta)}{\widetilde{Z}^s_5(\zeta)}$ is holomorphic outside the unit disk in the complex plane numerically. Let $r_1,~r_2,\cdots,r_5$ be the five roots of $\widetilde{Z}^s_5(\zeta)$ and set $r_{max} = \max\limits_{1\leq i \leq5}|r_i|$. We can get that
     \begin{equation}
         r_{max} < 0.96 <1,
     \end{equation}
     thus $\frac{\widetilde{A}^s_5(\zeta)}{\widetilde{Z}^s_5(\zeta)}$ is holomorphic outside the unit disk.
     
     Similarly, it follows from the maximum principle for harmonic function that $Re \frac{\widetilde{A}^s_5(\zeta)}{\widetilde{Z}^s_5(\zeta)} > 0$ for all $|\zeta| > 1$ is equivalent to 
     \begin{equation}\label{Re_5}
    Re[\widetilde{A}^s_5(e^{i\theta})\widetilde{Z}^s_5(e^{-i\theta})] \geq 0.
\end{equation}
Letting $y=cos \theta$ and using \eqref{Chebyshev}, we find that
\begin{equation}
    \aligned
    Re(\widetilde{A}^s_5(e^{i\theta})\widetilde{Z}_5^s(e^{-i\theta})) = (1-y)(\omega_4y^4 +\omega_3y^3  + \omega_2y^2+\omega_1y+\omega_0)=:(1-y)f(y),
    \endaligned
\end{equation}
with $w_4 >0, \ w_3<0, \ w_2>0, \ w_1<0, \ \omega_0>0$.

It is clear that \eqref{Re_5} is equivalent to 
\begin{equation}\label{Re_5*}
    \aligned
    f(y)\geq0,~\forall~y \in[-1,1].
    \endaligned
\end{equation}
With $\omega$ defined in \eqref{w-de4} and $s \geq 2$, we have
\begin{align}
    f(1) >0,~f(-1) >0,
\end{align}
and
\begin{equation}
    f'(y) = 4\omega_4y^3+3\omega_3y^2 + 2\omega_2y+\omega_1.
\end{equation}
If $f'(y)$ does not have zero in $[-1,1]$, then we can get \eqref{Re_4*}. Otherwise, supposing that there exists $-1 \leq y \leq 1$ such that $f'(y_0) = 0$, we only need to show $f(y_0)\geq 0$. Indeed, with $f'(y_0) = 0$, we have
\begin{equation}
    \aligned
    4 \omega_4f(y_0) = \omega_4(4f(y_0) - y_0f'(y_0)) - \frac{\omega_3}{4}f'(y_0),
    \endaligned
\end{equation}
then
\begin{equation}
    \aligned
    4 \omega_4f(y_0)= (2\omega_2\omega_4 - \frac{3}{4}\omega^2_3)y_0^2 + (3\omega_1\omega_4 - \frac{1}{2}\omega_2\omega_3)y_0 + 4 \omega_0\omega_4 - \frac{1}{4}\omega_1\omega_3,
    \endaligned
\end{equation}
and we set $h(x) = (2\omega_2\omega_4 - \frac{3}{4}\omega^2_3)x^2 + (3\omega_1\omega_4 - \frac{1}{2}\omega_2\omega_3)x + 4 \omega_0\omega_4 - \frac{1}{4}\omega_1\omega_3$ with $2\omega_2\omega_4 - \frac{3}{4}\omega^2_3 < 0$. 

Let $y_1$ and $y_2$ be the two roots of $f''(y) = 0$ and $x_1$ and $x_2$ be the two roots of $h(x) = 0$. We can obtain $x_1 < y_1 < y_2 <1 < x_2$ and $f'(x_1) < 0,~f'(y_1) > 0,~f'(y_2)>0,~f'(1) >0$. Therefore, $y_0 \in (x_1, 1)$ and $4\omega_4f(y_0) >0$, which implies $f(y_0) > 0$ with $\omega_4 >0$. Furthermore, we established that $Re(\widetilde{A}^s_5(e^{i\theta})\widetilde{Z}^s_5(e^{-i\theta}))\geq 0$ and by using Lemma \ref{definite G}, we can get \eqref{GBDFk} for $k=5$.
 \end{proof}

 \section{Numerical experiments}
In this section, we present a series of numerical experiments and comparisons involving semi-implicit GBDF schemes, fully explicit GBDF schemes and classical BDF schemes, with the aim of elucidating two key findings to validate our theoretical analysis: (i) Compared with the classical BDF schemes, the GBDF schemes admit a larger time step size. (ii) Under the same conditions, the GBDF schemes impose less restrictive constraints on the damping parameter \(\gamma\).  We emphasize that the theoretical minimum values of $\gamma$ listed in Tables 1.1 and 1.2 are generally pessimistic as indicated by our numerical results below.

For simplicity, we fix the computational domain \(\Omega = [0,\,2\pi)^2\) with periodic boundary conditions. The Fourier spectral method is used for spatial discretization by utilizing a \(64 \times 64\) grid. During the experiments, our primary focus is on  the time step size, damping parameters \(\gamma\) and exchange coefficient \(\beta\). In the subsequent tests, unless otherwise specified, the third-, fourth-, and fifth-order methods correspond to \(s=3\), \(s=4\), and \(s=5\), respectively.

\textbf{Example 1 (with a known exact solution).} We test the convergence rate for the LLG equation \eqref{e_original model} with an external force, so that the exact solution is
\begin{equation*}
   \setlength{\abovedisplayskip}{4pt} 
    \left\{
    \begin{array}{l}
         m_e^x(x,y,t) = \sin(t+x)\cos(t+y),  \\
         m_e^y(x,y,t) = \cos(t+x)\cos(t+y),  \\
         m_e^z(x,y,t) = \sin(t+y).
    \end{array}  \right. 
    \setlength{\belowdisplayskip}{4pt} 
\end{equation*}

\textbf{Convergence rate.} For the case of \(\gamma = 1\) and \(\beta = 0.5\), we plot the convergence rates of the $\|\mathbf{e}^n\|_{L^{\infty}(L^2)}$ and $\|\nabla \mathbf{e}^n\|_{L^{\infty}(L^2)}$ for the fully explicit GBDF scheme (with \(k=3,4,5\)) at \(T=0.5\) in Figs. \ref{fig1}(a)-\ref{fig1}(c), and those for the semi-implicit GBDF scheme (with \(k=3,4,5\)) at \(T=0.5\) in Figs. \ref{fig2}(a)-\ref{fig2}(c). The results obtained using the fully explicit GBDF scheme are consistent with Theorem \ref{the: error_estimate_final}, whereas those obtained with the semi-implicit GBDF scheme agree with Theorem \ref{the: error_estimate_final_semi}.

\begin{figure}[htbp]
	\centering
    \subfigure[GBDF3]
	{\includegraphics[width=0.3\linewidth, trim = {0cm 0cm 0cm 0cm}, clip]{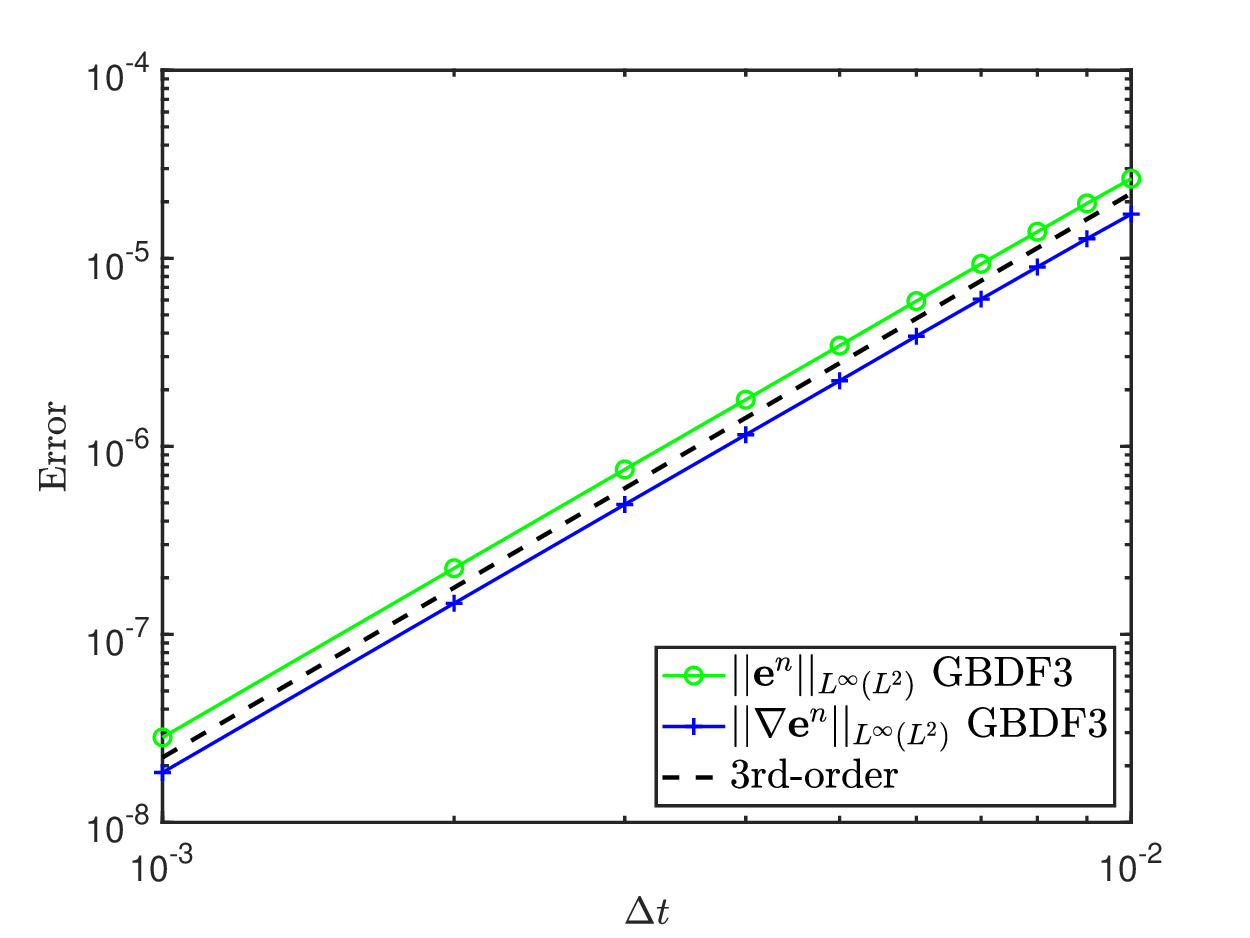}}
	\subfigure[GBDF4]
	{\includegraphics[width=0.3\linewidth, trim = {0cm 0cm 0cm 0cm}, clip]{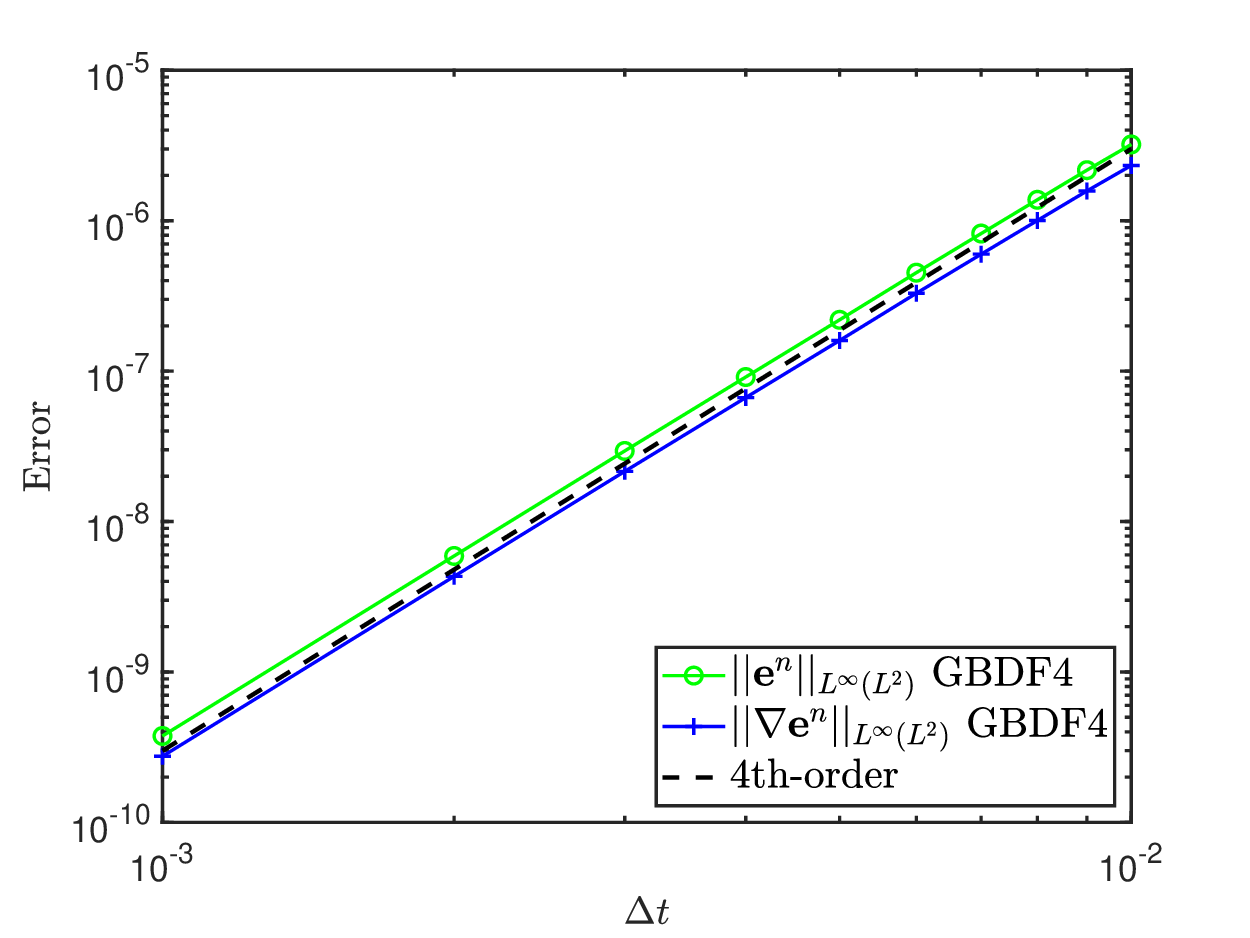}}
	\subfigure[GBDF5]
	{\includegraphics[width=0.3\linewidth, trim = {0cm 0cm 0cm 0cm}, clip]{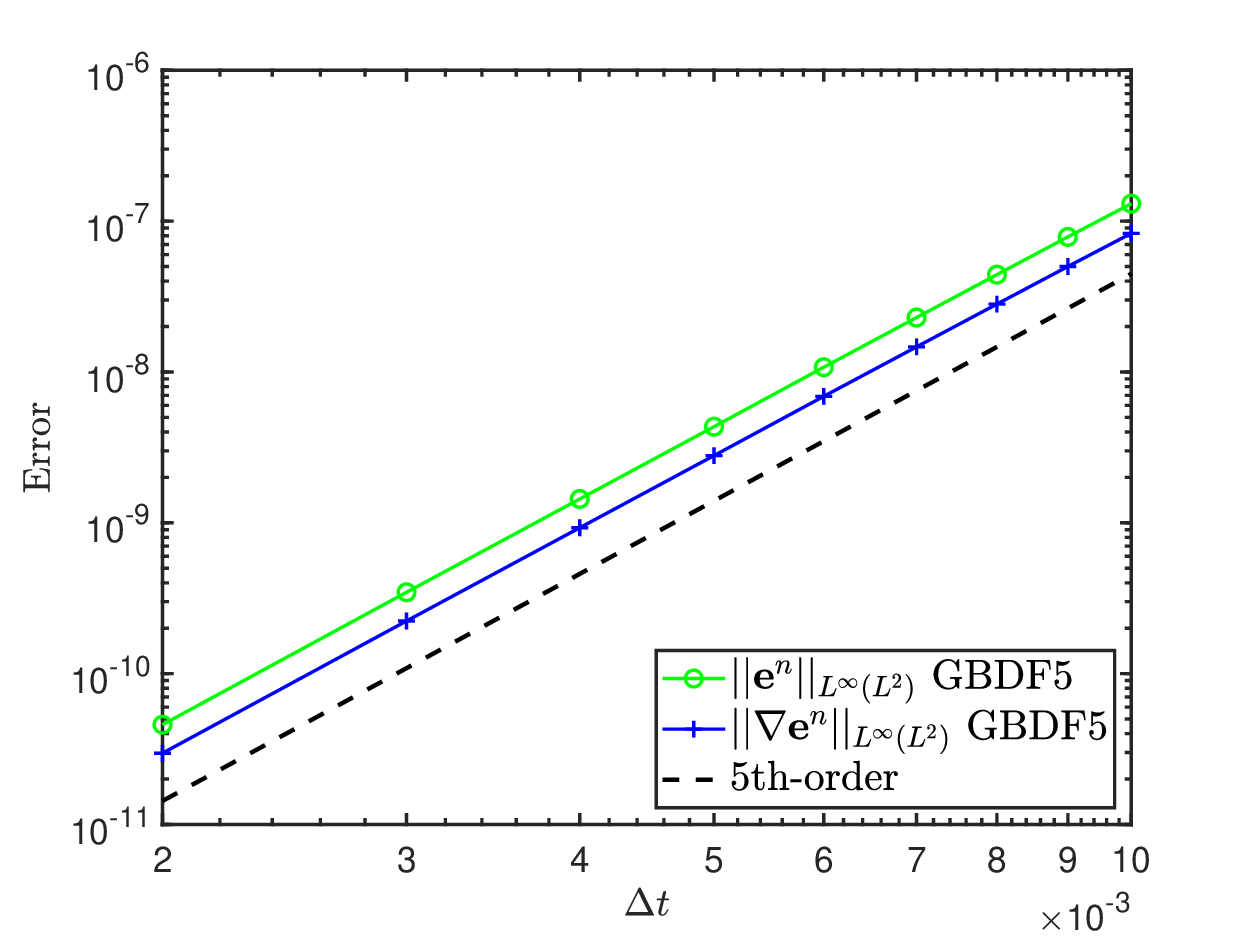}}
	\caption{Numerical convergence rate of the third- to fifth-order fully explicit GBDF schemes with $\gamma = 1,\,\beta = 0.5$}
	\label{fig1}
\end{figure}

\begin{figure}[htbp]
	\centering
    \subfigure[GBDF3]
	{\includegraphics[width=0.3\linewidth, trim = {0cm 0cm 0cm 0cm}, clip]{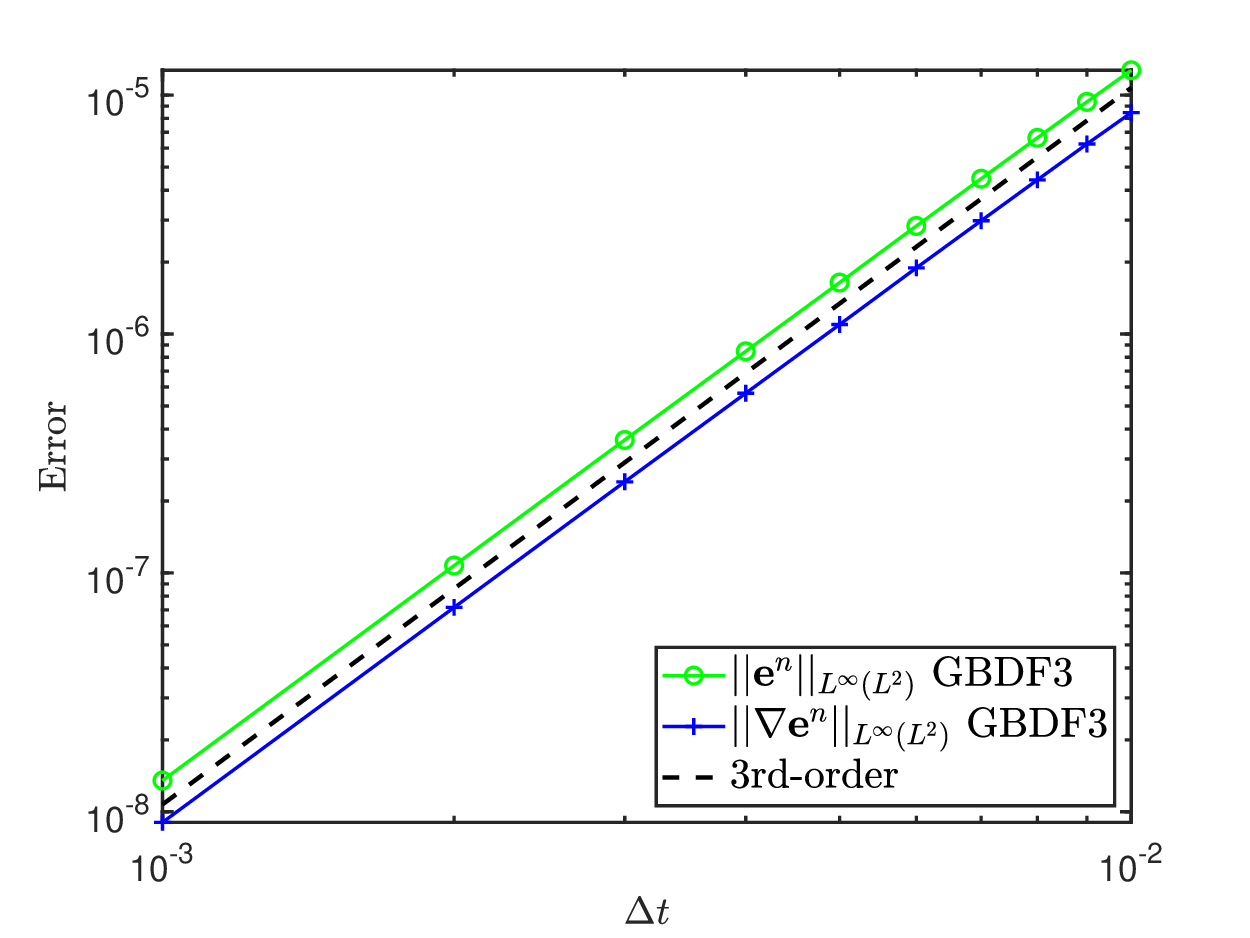}}
	\subfigure[GBDF4]
	{\includegraphics[width=0.3\linewidth, trim = {0cm 0cm 0cm 0cm}, clip]{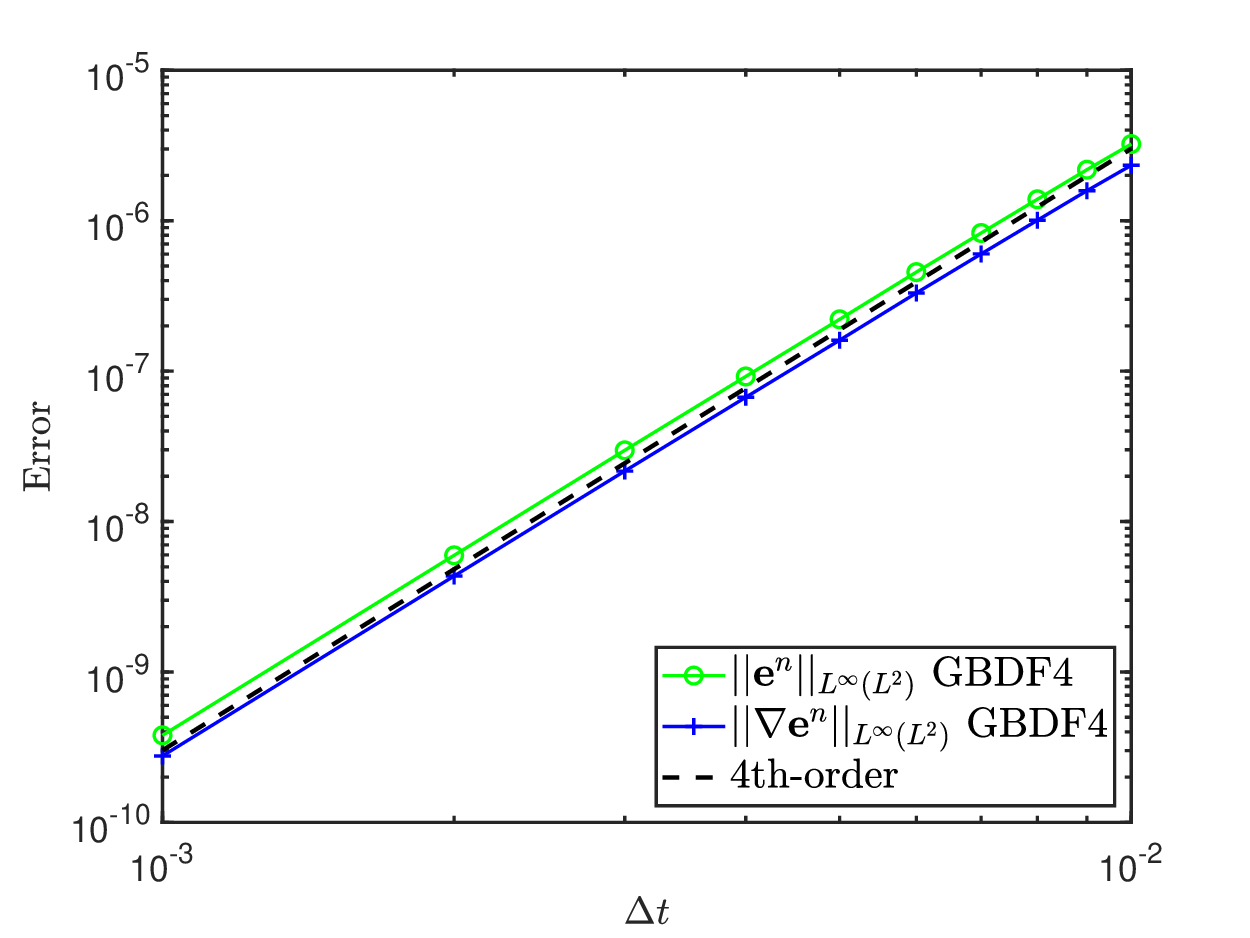}}
	\subfigure[GBDF5]
	{\includegraphics[width=0.3\linewidth, trim = {0cm 0cm 0cm 0cm}, clip]{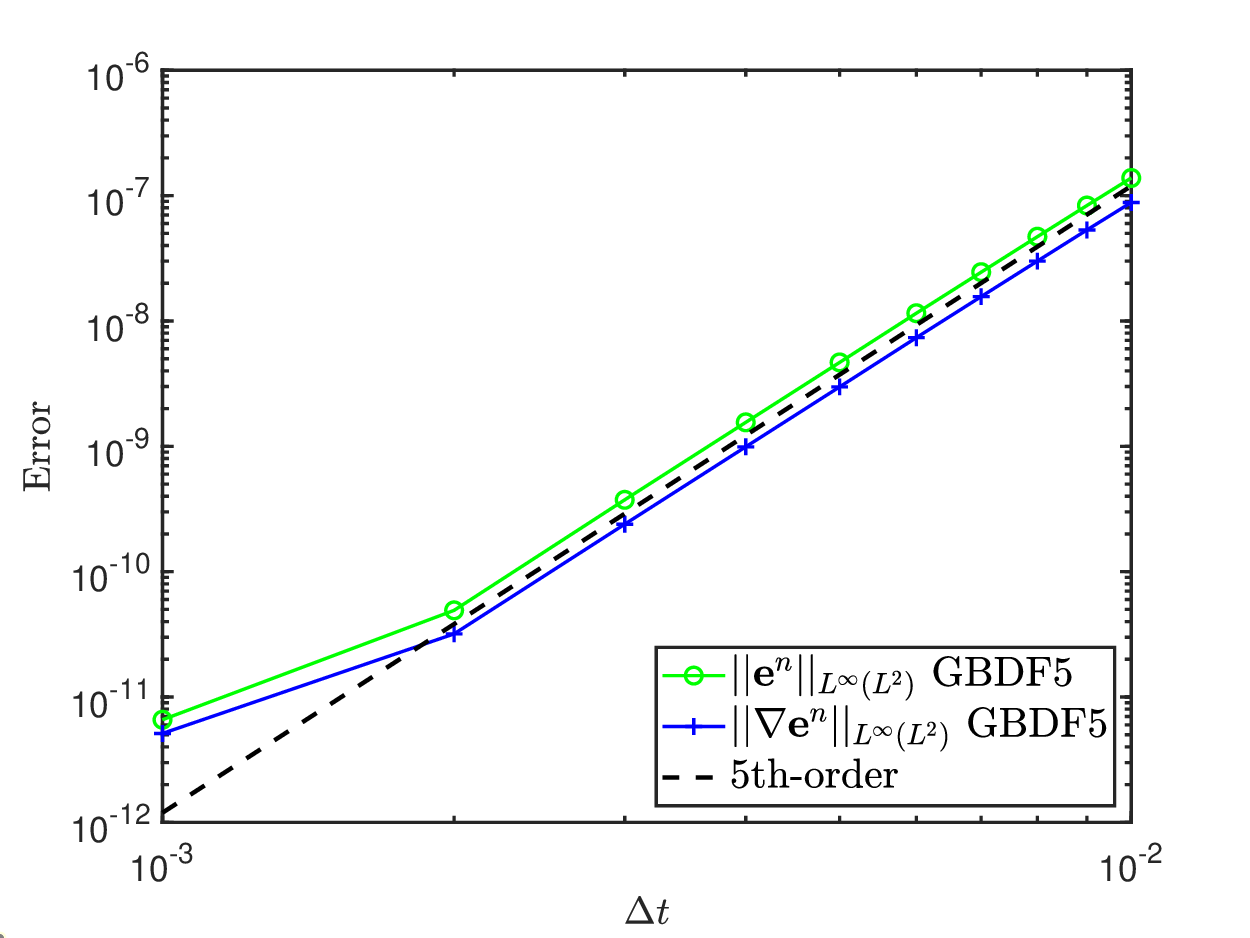}}
	\caption{Numerical convergence rate of the third- to fifth-order semi-implicit GBDF schemes with $\gamma = 1,\,\beta = 0.5$}
	\label{fig2}
\end{figure}

 In Table. \ref{table1}, we present a comparison between the fully explicit classical BDF schemes and the GBDF schemes, with fixed parameters \(\beta = 0.5\), \(\gamma = 1\), and \(T=10\). It is observed  that the GBDF schemes permit  larger time step sizes but with a slightly larger  truncation errors.  
 
 In Table. \ref{table4}, 
 we present a comparison between the semi-implicit classical BDF schemes and the GBDF schemes with \(\beta = 1\) and \(T=0.5\) by examining the influence of the damping parameter \(\gamma\).
 We observe that for the fourth- and fifth-order schemes, the GBDF schemes impose significantly less restrictions on  \(\gamma\). We also observe that the classical BDF schemes yield larger errors when $\gamma$ is small, which clearly reflects the insufficient stability of the classical BDF schemes.

\textbf{Example 2 (with an unknown exact solution).} We consider the Landau-Lifshitz equation \eqref{e_original model} with the initial condition
\begin{equation*}
   \setlength{\abovedisplayskip}{4pt} 
    \left\{
    \begin{array}{l}
         m_e^x(x,y,0) = \sin(x)\cos(y),  \\
         m_e^y(x,y,0) = \cos(x)\cos(y),  \\
         m_e^z(x,y,0) = \sin(y).
    \end{array}  \right. 
    \setlength{\belowdisplayskip}{4pt} 
\end{equation*}
Since the exact solution is unknown, 
 we use the semi-implicit fourth-order BDF scheme with time step $\Delta t=1\times10^{-6}$ to compute a reference solution.
In Figure. \ref{fig3}, we present the error comparison between the classical BDF schemes and the GBDF schemes with the same time step under the parameters $\gamma=1,\,\beta=0.5,\,T=10$, where Figure. \ref{fig3}(a) shows the comparison for third-order BDF and GBDF schemes and Figure. \ref{fig3}(b) for fourth-order BDF and GBDF schemes.
It can be observed that the GBDF schemes exhibit clear advantages.

In Table \ref{table3}, we fix $\Delta t=1\times 10^{-3}, \beta=0.5,$ and $T=2$, and  show the stability of different schemes with different $\gamma$. It can be observed that  the GBDF schemes have a much larger  allowable range of  $\gamma$ than that of the classical BDF schemes, while the allowable range of higher-order schemes is obviously smaller than that of lower-order schemes. 

We observe from the above results that semi-implicit schemes exhibit better stability than fully explicit schemes, but they typically incur a higher computational cost, while higher-order schemes often require smaller time steps than lower-order schemes in order to satisfy stability conditions. 
In Table \ref{table5}, we  fix $\gamma=1, \beta=0.5,$ and $T=10$, and compare the errors of different schemes with different  time step sizes. It is observed that   the  proposed higher-order fully explicit GBDF schemes not only demonstrate better stability than the classical lower-order semi-implicit schemes, but also achieve higher accuracy with lower computational cost. These results indicate that high-order fully explicit GBDF schemes may provide substantial advantages for numerical simulations of the LLG equation. 
\begin{table}[htbp]
	\centering
	\caption{Fully explicit GBDF vs. BDF schemes for Example 1 (\(\gamma = 1,\,\beta = 0.5,\,T=10\))}
	\label{table1}
	\small
	\begin{tabular}{c|c|c|c|c|c} \hline\hline
        Scheme&Time Step&\multicolumn{2}{c|}{\textbf{GBDF}} &\multicolumn{2}{c}{\textbf{Classical BDF}}\\ \hline
		&$\Delta t$&$||\mathbf{e}^n||_{L^{\infty}(L^{2})}$ &$||\nabla \mathbf{e}^n||_{L^{\infty}(L^{2})}$
		&$||\mathbf{e}^n||_{L^{\infty}(L^{2})}$ &$||\nabla \mathbf{e}^n||_{L^{\infty}(L^{2})}$  \\ \hline\hline
		\multirow{5}{*}{{\textbf{\large 3-order}}}
		&$1 \times10^{-2}$ &1.13E-2  &1.49E-2    &NaN      &NaN        \\
		&$8 \times10^{-3}$ &5.70E-3  &7.51E-3     &NaN     &NaN        \\
		&$6 \times10^{-3}$ &2.36E-3  &3.10E-3       &NaN      &NaN      \\
        &$4 \times10^{-3}$ &6.89E-4  &9.03E-4     &NaN      &NaN \\
        &$8 \times10^{-4}$ &5.33E-6  &6.97E-6     &6.67E-7     &8.26E-7         \\
		\hline\hline
        \multirow{5}{*}{{\textbf{\large 4-order}}}&$1 \times10^{-2}$ &1.68E-3  &2.26E-3       &NaN      &NaN          \\
		&$8 \times10^{-3}$ &6.99E-4  &9.41E-4     &NaN     &NaN         \\
		&$6 \times10^{-3}$ &2.24E-4  &3.01E-4    &NaN      &NaN        \\
		&$4 \times10^{-3}$ &4.52E-5  &6.07E-5     &NaN     &NaN        \\
		&$4 \times10^{-4}$ &5.27E-8  &7.02E-8     &2.06E-10      &2.74E-10         \\
		\hline\hline
        \multirow{5}{*}{{\textbf{\large 5-order}}}&$1 \times10^{-2}$ &5.11E-5  &6.67E-5       &NaN      &NaN          \\
		&$8 \times10^{-3}$ &1.62E-5  &2.11E-5     &NaN     &NaN         \\
		&$6 \times10^{-3}$ &3.70E-6  &4.78E-6     &NaN      &NaN         \\
        &$4 \times10^{-3}$ &4.72E-7  &6.10E-7    &NaN      &NaN       \\
		&$2 \times10^{-4}$ &5.25E-8  &6.93E-8     &2.99E-10     &3.91E-10        \\
        \hline\hline
	\end{tabular}
\end{table}

\begin{table}[htbp]
	\centering
	\caption{Semi-implicit GBDF vs. BDF schemes for Example 1 with different $\gamma$ (\(\,\beta = 1,\,T=0.5\))}
	\label{table4}
	\small
    \begin{tabular}{c|c|c|c|c|c} \hline\hline
        Scheme&&\multicolumn{2}{c|}{\textbf{GBDF}} &\multicolumn{2}{c}{\textbf{Classical BDF}}\\ \hline
		&$\gamma$&$||\mathbf{e}^n||_{L^{\infty}(L^{2})}$ &$||\nabla \mathbf{e}^n||_{L^{\infty}(L^{2})}$
		&$||\mathbf{e}^n||_{L^{\infty}(L^{2})}$ &$||\nabla \mathbf{e}^n||_{L^{\infty}(L^{2})}$  \\ \hline\hline
        \multirow{3}{*}{{\textbf{\large 4-order}}}&$\gamma = 0.1$ &4.87E-5  &1.14E-4       &NaN      &NaN          \\
		&$\gamma = 0.2$ &4.37E-11  &6.50E-11     &2.51E-6     &2.90E-6         \\
		 \multirow{1}{*}{{$\Delta t = 7\times 10^{-4}$}}&$\gamma = 0.3$ &4.61E-11  &6.72E-11      &6.95E-7  &8.36E-7      \\
       &$\gamma = 0.5$ &5.10E-11  &7.27E-11     &1.30E-12     &1.95E-12         \\
		\hline\hline
        \multirow{3}{*}{{\textbf{\large 5-order}}}&$\gamma = 0.4$ &6.57E-12  &1.08E-11       &NaN      &NaN          \\
        &$\gamma = 0.5$ &7.15E-12  &1.09E-11     &NaN      &NaN         \\
		\multirow{1}{*}{{$\Delta t = 8\times10^{-4}$}}&$\gamma = 0.6$ &7.54E-12  &1.10E-11     &2.95E-6      &3.39E-6         \\
        &$\gamma = 1.2$&7.41E-12  &9.74E-12     &1.02E-13  &7.82E-14        \\
        \hline\hline
	\end{tabular}
\end{table}

\begin{figure}[htbp]
	\centering
    \subfigure[BDF3 v.s. GBDF3]
	{\includegraphics[width=0.4\linewidth, trim = {0cm 0cm 0cm 0cm}, clip]{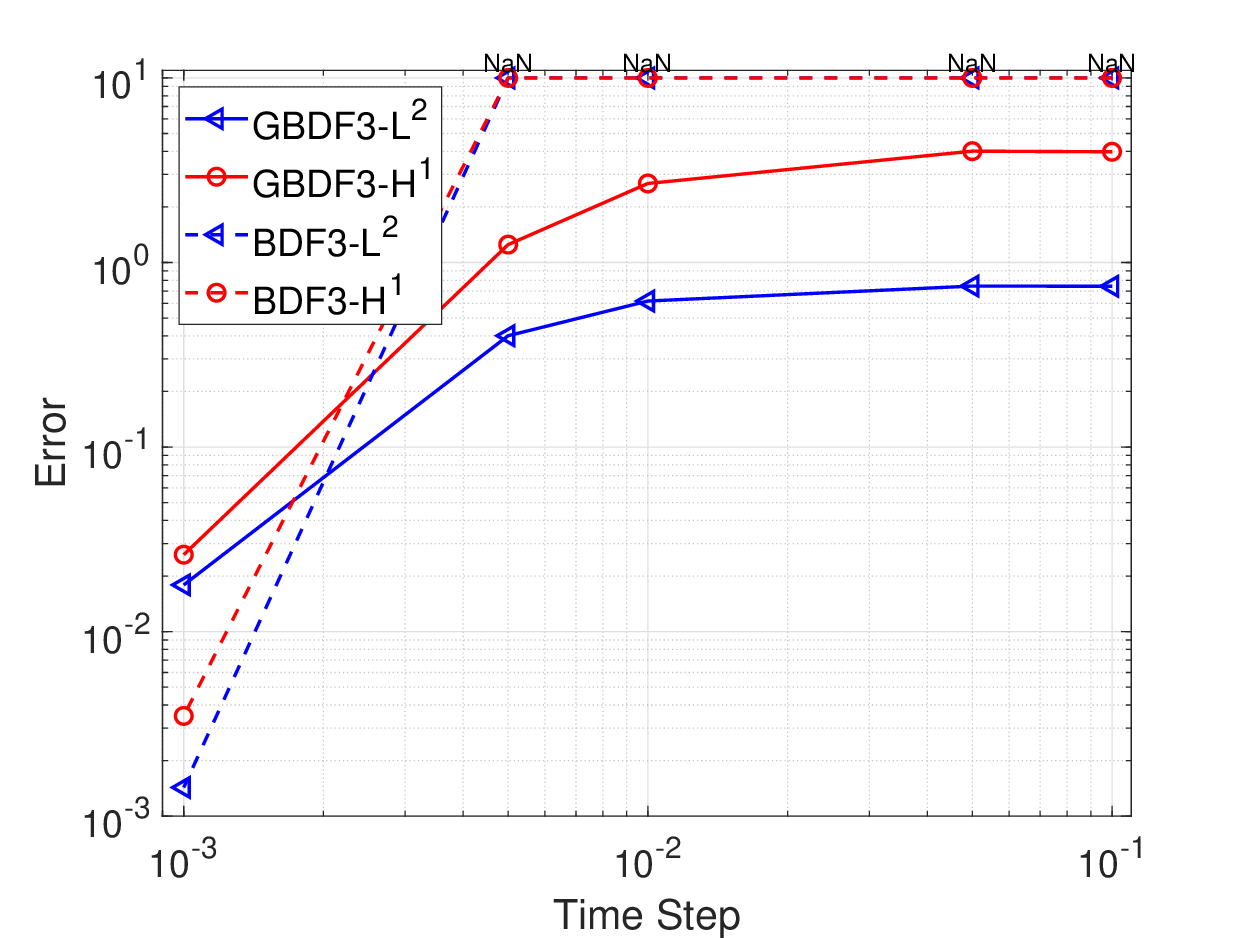}}
	\subfigure[BDF4 v.s. GBDF4]
	{\includegraphics[width=0.4\linewidth, trim = {0cm 0cm 0cm 0cm}, clip]{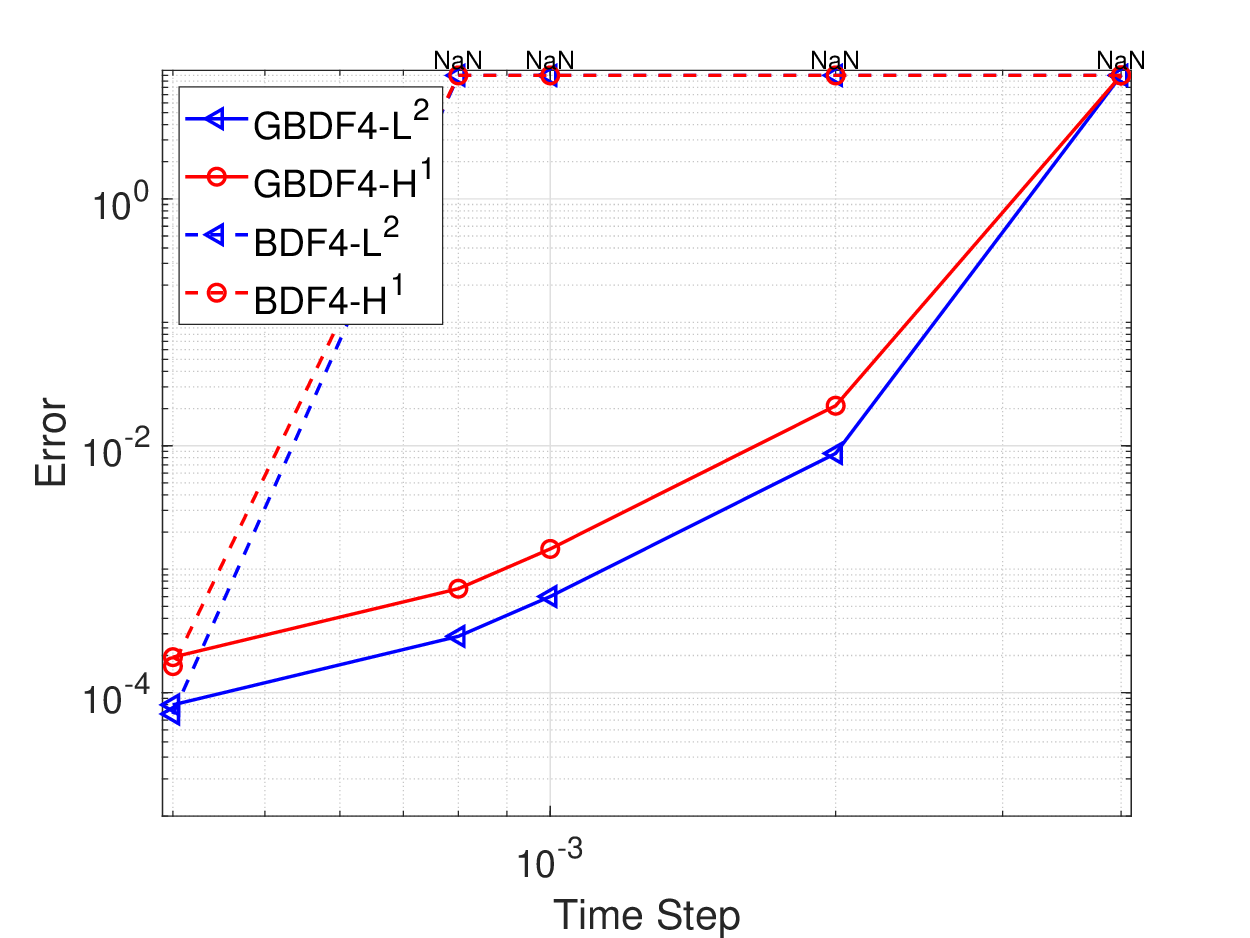}}
	\caption{BDF v.s. GBDF schemes for Example 2 with $\gamma = 1,\,\beta = 0.5, \,T=10$}
	\label{fig3}
\end{figure}

\begin{table}[htbp]
	\centering
	\caption{Different $\gamma$ with BDF and GBDF schemes for Example 2 (\(\Delta t = 1e-3,\,\beta = 0.5,\,T=2\))}
	\label{table3}
	\small
	\begin{tabular}{c|c|c|c|c|c} \hline\hline
        Scheme & $\gamma=1.0$ & $\gamma = 0.8$ & $\gamma =0.6$ & $\gamma = 0.4$ & $\gamma = 0.3$\\ \hline
        BDF3 & \ding{51} & \ding{55} & \ding{55} & \ding{55} & \ding{55}\\ \hline
        GBDF3 & \ding{51} & \ding{51} & \ding{51} & \ding{51} & \ding{55}\\ \hline
        BDF4 & \ding{55} & \ding{55} & \ding{55} & \ding{55} & \ding{55}\\ \hline
        GBDF4 & \ding{51} & \ding{51} & \ding{51} & \ding{55} & \ding{55}\\
        \hline\hline
	\end{tabular}

    \par
    \vspace{0.5em}
    \footnotesize
    \textit{Note:} \ding{51}-Success,\,\,\ding{55}-NaN
\end{table}

\begin{table}[htbp]
	\centering
	\caption{Semi-implicit BDF2 scheme v.s. Fully explicit GBDF3 and BDF4 schemes for Example 2 (\(\gamma  = 1,\,\beta = 0.5,\,T=10\))}
	\label{table5}
	\small
	\begin{tabular}{c|c|c|c|c|c|c} \hline\hline
        Time Step&\multicolumn{2}{c|}{\textbf{Semi-implicit BDF2}} &\multicolumn{2}{c|}{\textbf{Fully explicit GBDF3}}&\multicolumn{2}{c}{\textbf{Fully explicit GBDF4}}\\ \hline
		$\Delta t$&$||\mathbf{e}^n||_{L^{\infty}(L^{2})}$ &$||\nabla \mathbf{e}^n||_{L^{\infty}(L^{2})}$
		&$||\mathbf{e}^n||_{L^{\infty}(L^{2})}$ &$||\nabla \mathbf{e}^n||_{L^{\infty}(L^{2})}$ &$||\mathbf{e}^n||_{L^{\infty}(L^{2})}$ &$||\nabla \mathbf{e}^n||_{L^{\infty}(L^{2})}$ \\ \hline\hline
		$4 \times10^{-3}$ &NaN  &NaN       &3.18E-1      &9.16E-1 &NaN &NaN          \\
		$2 \times10^{-3}$ &NaN  &NaN     &1.03E-1     &2.59E-1 &8.67E-3 &2.11E-2         \\
		$1 \times10^{-3}$ &NaN  &NaN    &1.79E-2      &4.37E-2 & 6.01E-4 &1.46E-3       \\
        $4 \times10^{-4}$ &3.93E-2  &9.48E-2    &1.16E-3      &2.82E-3   &7.97E-5 &1.94E-4   \\
         $2 \times10^{-4}$ &8.87E-3  &2.15E-2    &8.54E-5      &2.08E-4  &6.68E-5 &1.63E-4      \\
		\hline\hline
	\end{tabular}
\end{table}

\section{Concluding remarks}
In this paper, we developed a class of high-order semi-implicit and fully explicit generalized backward differentiation formula (GBDF) schemes for the Landau-Lifshitz-Gilbert equation. The primary challenge in the analysis of high-order time discretizations for this problem lies in the loss of cancellation in the skew-symmetric trilinear term at the discrete level, which in existing approaches leads to restrictive lower bounds on the damping parameter.

To overcome this difficulty, we introduced a novel multiplier that restores a suitable structure in the discrete energy analysis. This new analytical mechanism enables us to control the non-vanishing trilinear terms and to establish optimal-order error estimates under substantially relaxed assumptions on the damping parameter. In particular, for fourth- and fifth-order schemes, the proposed approach significantly enlarges the admissible parameter regimes and permits the use of larger time steps compared with classical BDF-based methods.

The analysis presented here provides a unified framework for both semi-implicit and fully explicit treatments of the gyromagnetic term. More broadly, it offers a new perspective on the treatment of skew-symmetric nonlinearities in high-order multistep schemes, which may be of independent interest beyond the LLG equation.

An important direction is the development of  fully discrete schemes that retain the main advantages of the proposed time-discrete schemes and preserve additional geometric features. It would also be of interest to extend the proposed framework to  other nonlinear evolution equations with similar structural properties.  Finally, further investigation of the practical performance of the proposed methods in large-scale simulations and applications to complex magnetic systems would be valuable.

\bibliographystyle{plain}
\bibliography{ref}

\begin{thebibliography}{10}

\bibitem{akrivis2021higher}
Georgios Akrivis, Michael Feischl, Bal{\'a}zs Kov{\'a}cs, and Christian Lubich.
\newblock Higher-order linearly implicit full discretization of the {L}andau--{L}ifshitz--{G}ilbert equation.
\newblock {\em Mathematics of Computation}, 90(329):995--1038, 2021.

\bibitem{alouges2006convergence}
Fran{\c{c}}ois Alouges and Pascal Jaisson.
\newblock Convergence of a finite element discretization for the {L}andau--{L}ifshitz equations in micromagnetism.
\newblock {\em Mathematical Models and Methods in Applied Sciences}, 16(02):299--316, 2006.

\bibitem{an2021optimal}
Rong An, Huadong Gao, and Weiwei Sun.
\newblock Optimal error analysis of {E}uler and {C}rank--{N}icolson projection finite difference schemes for {L}andau--{L}ifshitz equation.
\newblock {\em SIAM Journal on Numerical Analysis}, 59(3):1639--1662, 2021.

\bibitem{an2025optimal}
Rong An, Yonglin Li, and Weiwei Sun.
\newblock Optimal error analysis of the normalized tangent plane {FEM} for {L}andau--{L}ifshitz--{G}ilbert equation.
\newblock {\em IMA Journal of Numerical Analysis}, 45(5):3109--3137, 2025.

\bibitem{cai2022second}
Yongyong Cai, Jingrun Chen, Cheng Wang, and Changjian Xie.
\newblock A second-order numerical method for {L}andau-{L}ifshitz-{G}ilbert equation with large damping parameters.
\newblock {\em Journal of Computational Physics}, 451:110831, 2022.

\bibitem{cervera2007numerical}
Carlos J~Garc{\'\i}a Cervera.
\newblock Numerical micromagnetics: a review.
\newblock {\em SeMA Journal: Bolet{\'\i}n de la Sociedad Espa{\~n}ola de Matem{\'a}tica Aplicada}, (39):103--136, 2007.

\bibitem{cimrak2007survey}
Ivan Cimr{\'a}k.
\newblock A survey on the numerics and computations for the landau-lifshitz equation of micromagnetism.
\newblock {\em Archives of Computational Methods in Engineering}, 15(3):1--37, 2007.

\bibitem{cohen1989relaxation}
Robert Cohen, San-Yih Lin, and Mitchell Luskin.
\newblock Relaxation and gradient methods for molecular orientation in liquid crystals.
\newblock {\em Computer Physics Communications}, 53(1-3):455--465, 1989.

\bibitem{dahlquist1978g}
Germund Dahlquist.
\newblock {G}-stability is equivalent to {A}-stability.
\newblock {\em BIT Numerical Mathematics}, 18:384--401, 1978.

\bibitem{gui2022convergence}
Xinping Gui, Buyang Li, and Jilu Wang.
\newblock Convergence of renormalized finite element methods for heat flow of harmonic maps.
\newblock {\em SIAM Journal on Numerical Analysis}, 60(1):312--338, 2022.

\bibitem{gui2024implicit}
Yan Gui, Cheng Wang, and Jingrun Chen.
\newblock Implicit-explicit {R}unge-{K}utta methods for {L}andau-{L}ifshitz equation with arbitrary damping.
\newblock {\em Communications in Mathematical Sciences}, 2024.

\bibitem{he2024temporal}
Jiayun He, Lei Yang, and Jiajun Zhan.
\newblock Temporal high-order accurate numerical scheme for the {L}andau--{L}ifshitz--{G}ilbert equation.
\newblock {\em Mathematics}, 12(8):1179, 2024.

\bibitem{HeSu07}
Yinnian He and Weiwei Sun.
\newblock Stability and convergence of the {C}rank--{N}icolson/{A}dams--{B}ashforth scheme for the time-dependent {N}avier--{S}tokes equations.
\newblock {\em SIAM Journal on Numerical Analysis}, 45(2):837--869, 2007.

\bibitem{huang2024new}
Fukeng Huang and Jie Shen.
\newblock On a new class of {BDF} and {IMEX} schemes for parabolic type equations.
\newblock {\em SIAM Journal on Numerical Analysis}, 62(4):1609--1637, 2024.

\bibitem{huang2025stability}
Fukeng Huang and Jie Shen.
\newblock Stability and error analysis of a new class of higher-order consistent splitting schemes for the navier-stokes equations.
\newblock {\em Mathematics of Computation}, 2025.

\bibitem{huang2003high}
Zhongyi Huang.
\newblock High accuracy numerical method of thin-film problems in micromagnetics.
\newblock {\em Journal of Computational Mathematics}, pages 33--40, 2003.

\bibitem{jia2025electrically}
Junhui Jia et~al.
\newblock Electrically tuning photonic topological quasiparticles in synthetic two-level system.
\newblock {\em Nature Physics}, pages 1--8, 2025.

\bibitem{jiang2015blowing}
Wanjun Jiang et~al.
\newblock Blowing magnetic skyrmion bubbles.
\newblock {\em Science}, 349(6245):283--286, 2015.

\bibitem{kim2017mimetic}
Eugenia Kim and Konstantin Lipnikov.
\newblock The mimetic finite difference method for the landau--lifshitz equation.
\newblock {\em Journal of Computational Physics}, 328:109--130, 2017.

\bibitem{kruzik2006recent}
Martin Kruzik and Andreas Prohl.
\newblock Recent developments in the modeling, analysis, and numerics of ferromagnetism.
\newblock {\em SIAM review}, 48(3):439--483, 2006.

\bibitem{landau1935theory}
LALE Landau, Evgeny Lifshitz, et~al.
\newblock On the theory of the dispersion of magnetic permeability in ferromagnetic bodies.
\newblock {\em Phys. Z. Sowjetunion}, 8(153):101--114, 1935.

\bibitem{li2026class}
Xiaoli Li, Jie Shen, and Nan Zheng.
\newblock On a class of higher-order length preserving and energy decreasing {IMEX} schemes for the {L}andau-{L}ifshitz equation.
\newblock {\em Journal of Computational Physics}, page 114786, 2026.

\bibitem{nagaosa2013topological}
Naoto Nagaosa and Yoshinori Tokura.
\newblock Topological properties and dynamics of magnetic skyrmions.
\newblock {\em Nature {N}anotechnology}, 8(12):899--911, 2013.

\bibitem{prohl2001computational}
Andreas Prohl et~al.
\newblock {\em Computational micromagnetism}.
\newblock Springer, 2001.

\bibitem{romming2013writing}
Niklas Romming, Christian Hanneken, Matthias Menzel, Jessica~E Bickel, Boris Wolter, Kirsten Von~Bergmann, Andr{\'e} Kubetzka, and Roland Wiesendanger.
\newblock Writing and deleting single magnetic skyrmions.
\newblock {\em Science}, 341(6146):636--639, 2013.

\bibitem{shen1990long}
Jie Shen.
\newblock Long time stability and convergence for fully discrete nonlinear {G}alerkin methods.
\newblock {\em Applicable Analysis}, 38(4):201--229, 1990.

\bibitem{suess2002time}
Dieter Suess, Vassilios Tsiantos, Thomas Schrefl, Josef Fidler, Werner Scholz, Hermann Forster, Rok Dittrich, and James~J Miles.
\newblock Time resolved micromagnetics using a preconditioned time integration method.
\newblock {\em Journal of Magnetism and Magnetic Materials}, 248(2):298--311, 2002.

\bibitem{xie2025error}
Changjian Xie and Cheng Wang.
\newblock Error analysis of third-order in time and fourth-order linear finite difference scheme for {L}andau-{L}ifshitz-{G}ilbert equation under large damping parameters.
\newblock {\em arXiv preprint arXiv:2510.25172}, 2025.

\end{thebibliography}

\end{document}